\documentclass[12pt]{amsart}
\usepackage{amsmath,amsthm,amsfonts,amssymb,mathrsfs}
\date{\today}

\input xymatrix
\xyoption{all}

\usepackage{color}

\newtheorem{theorem}{Theorem}[section]
\newtheorem{proposition}[theorem]{Proposition}
\newtheorem{corollary}[theorem]{Corollary}
\newtheorem{lemma}[theorem]{Lemma}

\theoremstyle{definition}
\newtheorem{example}[theorem]{Example}
\newtheorem{remark}[theorem]{Remark}

\usepackage{hyperref}

\begin{document}

\title[On monoids of monotone injective partial
selfmaps of $L_n\times_{\operatorname{lex}}\mathbb{Z}$]{On monoids of monotone injective partial selfmaps of $L_n\times_{\operatorname{lex}}\mathbb{Z}$ with co-finite domains and images}

\author[O.~Gutik and I.~Pozdnyakova]{Oleg~Gutik and Inna Pozdnyakova}
\address{Faculty of Mechanics and Mathematics, Ivan Franko
National University of Lviv, Universytetska 1, Lviv, 79000, Ukraine}
\email{o\_gutik@franko.lviv.ua, ovgutik@yahoo.com, pozdnyakova.inna@gmail.com}



\begin{abstract}
We study the semigroup $\mathscr{I\!O}\!_{\infty}(\mathbb{Z}^n_{\operatorname{lex}})$ of monotone injective partial selfmaps of the set of $L_n\times_{\operatorname{lex}}\mathbb{Z}$ having co-finite domain and image, where $L_n\times_{\operatorname{lex}}\mathbb{Z}$ is the lexicographic product of $n$-elements chain and the set of integers with the usual order. We show that $\mathscr{I\!O}\!_{\infty}(\mathbb{Z}^n_{\operatorname{lex}})$ is bisimple and establish its projective congruences. We prove that $\mathscr{I\!O}\!_{\infty}(\mathbb{Z}^n_{\operatorname{lex}})$ is finitely generated, and for $n=1$ every automorphism of $\mathscr{I\!O}\!_{\infty}(\mathbb{Z}^n_{\operatorname{lex}})$ is inner and show that  in the case $n\geqslant 2$ the semigroup $\mathscr{I\!O}\!_{\infty}(\mathbb{Z}^n_{\operatorname{lex}})$ has non-inner automorphisms. Also we show that every Baire topology $\tau$ on $\mathscr{I\!O}\!_{\infty}(\mathbb{Z}^n_{\operatorname{lex}})$
such that $(\mathscr{I\!O}\!_{\infty}(\mathbb{Z}^n_{\operatorname{lex}}),\tau)$ is
a Hausdorff semitopological semigroup is discrete, construct a non-discrete Hausdorff semigroup inverse topology on $\mathscr{I\!O}\!_{\infty}(\mathbb{Z}^n_{\operatorname{lex}})$, and prove that the discrete semigroup
$\mathscr{I\!O}\!_{\infty}(\mathbb{Z}^n_{\operatorname{lex}})$ cannot be embedded
into some classes of compact-like topological semigroups and that its remainder under the closure in a topological semigroup $S$ is an ideal in $S$.
\end{abstract}

\maketitle


\section{Introduction and preliminaries}

In this paper all spaces will be assumed to be Hausdorff. We shall
denote the first infinite cardinal by $\omega$ and the cardinality
of the set $A$ by $|A|$. Also we denote the additive group of
integers by $\mathbb{Z}(+)$. We shall identify all sets $X$ with its
cardinality $|X|$.

An algebraic semigroup $S$ is called {\it inverse} if for any
element $x\in S$ there exists a unique $x^{-1}\in S$ such that
$xx^{-1}x=x$ and $x^{-1}xx^{-1}=x^{-1}$. The element $x^{-1}$ is
called the {\it inverse of} $x\in S$. If $S$ is an inverse
semigroup, then the function $\operatorname{inv}\colon S\to S$ which
assigns to every element $x$ of $S$ its inverse element $x^{-1}$ is
called an {\it inversion}.

If $\mathfrak{C}$ is an arbitrary congruence on a semigroup $S$,
then we denote by $\Phi_\mathfrak{C}\colon S\rightarrow
S/\mathfrak{C}$ the natural homomorphisms from $S$ onto the quotient
semigroup $S/\mathfrak{C}$. A congruence $\mathfrak{C}$ on a
semigroup $S$ is called \emph{non-trivial} if $\mathfrak{C}$ is
distinct from universal and identity congruence on $S$, and
\emph{group} if the quotient semigroup $S/\mathfrak{C}$ is a group.
Every inverse semigroup $S$ admits a least (minimum) group
congruence $\sigma$:
\begin{equation*}
    a\sigma b \; \hbox{ if and only if there exists }\;
    e\in E(S) \; \hbox{ such that }\; ae=be
\end{equation*}
(see \cite[Lemma~III.5.2]{Petrich1984})

If $S$ is a semigroup, then we shall denote the subset of
idempotents in $S$ by $E(S)$. If $S$ is an inverse semigroup, then
$E(S)$ is closed under multiplication and we shall refer to $E(S)$ a
\emph{band} (or the \emph{band of} $S$). If the band $E(S)$ is a
non-empty subset of $S$, then the semigroup operation on $S$
determines the following partial order $\leqslant$ on $E(S)$:
$e\leqslant f$ if and only if $ef=fe=e$. This order is called the
{\em natural partial order} on $E(S)$. A \emph{semilattice} is a
commutative semigroup of idempotents. A semilattice $E$ is called
{\em linearly ordered} or a \emph{chain} if its natural order is a
linear order. A \emph{maximal chain} of a semilattice $E$ is a chain
which is properly contained in no other chain of $E$.

The Axiom of Choice implies the existence of maximal chains in any
partially ordered set. According to
\cite[Definition~II.5.12]{Petrich1984} a chain $L$ is called an
$\omega$-chain if $L$ is isomorphic to $\{0,-1,-2,-3,\ldots\}$ with
the usual order $\leqslant$. Let $E$ be a semilattice and $e\in E$.
We denote ${\downarrow} e=\{ f\in E\mid f\leqslant e\}$ and
${\uparrow} e=\{ f\in E\mid e\leqslant f\}$.  By
$(\mathscr{P}_{<\omega}(\lambda),\subseteq)$ we shall denote the
\emph{free semilattice with identity} over a set of cardinality
$\lambda\geqslant\omega$, i.e.,
$(\mathscr{P}_{<\omega}(\lambda),\subseteq)$ is the set of all
finite subsets (with the empty set) of $\lambda$ with the
semilattice operation ``union''.

If $S$ is a semigroup, then we shall denote the Green relations on
$S$ by $\mathscr{R}$, $\mathscr{L}$, $\mathscr{J}$, $\mathscr{D}$
and $\mathscr{H}$ (see \cite[Section~2.1]{CP}):
\begin{align*}
    &\qquad a\mathscr{R}b \mbox{ if and only if } aS^1=bS^1;\\
    &\qquad a\mathscr{L}b \mbox{ if and only if } S^1a=S^1b;\\
    &\qquad a\mathscr{J}b \mbox{ if and only if } S^1aS^1=S^1bS^1;\\
    &\qquad \mathscr{D}=\mathscr{L}\circ\mathscr{R}=
          \mathscr{R}\circ\mathscr{L};\\
    &\qquad \mathscr{H}=\mathscr{L}\cap\mathscr{R}.
\end{align*}
A semigroup $S$ is called \emph{simple} if $S$ does not contain any
proper two-sided ideals and \emph{bisimple} if $S$ has a unique $\mathscr{D}$-class.

For a non-empty subset $A$ of an inverse semigroup $S$ we say that $A$ \emph{generates} $S$ as an inverse semigroup, if the intersection of all inverse subsemigroups of $S$ whose contains $A$ coincides with $S$. In this case we write $\langle A\rangle=S$ and call $A$ to be a \emph{set of generators of} $S$ as an inverse semigroup.

An automorphism $\mathfrak{f}\colon S\rightarrow S$ of a semigroup $S$ with a non-empty group of units $H_1$ is called \emph{inner} if there exists $a\in H_1$ such that $(s)\mathfrak{f}=asa^{-1}$ for all $s\in S$.

A {\it semitopological} (resp. \emph{topological}) {\it semigroup}
is a Hausdorff topological space together with a separately (resp.
jointly) continuous semigroup operation. An inverse topological
semigroup with the continuous inversion is called a
\emph{topological inverse semigroup}. A Hausdorff topology $\tau$ on
a (inverse) semigroup $S$ such that $(S,\tau)$ is a topological
(inverse) semigroup is called a (\emph{inverse}) \emph{semigroup
topology}.

If $\alpha\colon X\rightharpoonup Y$ is a partial map, then by
$\operatorname{dom}\alpha$ and $\operatorname{ran}\alpha$ we denote
the domain and the range of $\alpha$, respectively.

Let $\mathscr{I}_\lambda$ denote the set of all partial one-to-one
transformations of an infinite set $X$ of cardinality $\lambda$
together with the following semigroup operation:
$x(\alpha\beta)=(x\alpha)\beta$ if
$x\in\operatorname{dom}(\alpha\beta)=\{
y\in\operatorname{dom}\alpha\mid
y\alpha\in\operatorname{dom}\beta\}$,  for
$\alpha,\beta\in\mathscr{I}_\lambda$. The semigroup
$\mathscr{I}_\lambda$ is called the \emph{symmetric inverse
semigroup} over the set $X$~(see \cite[Section~1.9]{CP}). The
symmetric inverse semigroup was introduced by
Vagner~\cite{Wagner1952} and it plays a major role in the theory of
semigroups. An element $\alpha\in\mathscr{I}_\lambda$ is called \emph{cofinite}, if the sets $\lambda\setminus\operatorname{dom}\alpha$ and $\lambda\setminus\operatorname{ran}\alpha$ are finite.

Let $(X,\leqslant)$ be a partially ordered set. We shall say that a partial map
$\alpha\colon X\rightharpoonup X$ is \emph{monotone} if $x\leqslant y$ implies $(x)\alpha\leqslant(y)\alpha$ for $x,y\in X$.

Let $\mathbb{Z}$ be the set of integers with the usual linear order
$\le$. For any positive integer $n$ by $L_n$ we denote the set $\{1,\ldots,n\}$ with the usual linear order $\le$. On the Cartesian product $L_n\times\mathbb{Z}$ we define the lexicographic order, i.e.,
\begin{equation*}
    (i,m)\leqslant(j,n) \qquad \hbox{if and only if} \qquad (i<j) \quad \hbox{or} \quad (i=j \quad\hbox{and} \quad m\le n).
\end{equation*}
Later the set $L_n\times\mathbb{Z}$ with the lexicographic order we denote by $L_n\times_{\operatorname{lex}}\mathbb{Z}$. Also, it is obvious that the $\mathbb{Z}\times L_n$ with the lexicographic order is order isomorphic to $(\mathbb{Z},\le)$.

By $\mathscr{I\!O}\!_{\infty}(\mathbb{Z}^n_{\operatorname{lex}})$ we denote a subsemigroup of injective partial monotone selfmaps of $L_n\times_{\operatorname{lex}}\mathbb{Z}$ with co-finite domains and images.
Obviously, $\mathscr{I\!O}\!_{\infty}(\mathbb{Z}^n_{\operatorname{lex}})$ is an
inverse submonoid of the semigroup $\mathscr{I}_\omega$ and $\mathscr{I\!O}\!_{\infty}(\mathbb{Z}^n_{\operatorname{lex}})$ is a countable semigroup. Also, by $\mathscr{I\!O}\!_{\infty}(\mathbb{Z})$ we denote a subsemigroup of injective partial monotone selfmaps of $\mathbb{Z}$ with cofinite domains and images.

Furthermore, we shall denote the identity of the semigroup
$\mathscr{I\!O}\!_{\infty}(\mathbb{Z}^n_{\operatorname{lex}})$ by $\mathbb{I}$ and
the group of units of
$\mathscr{I\!O}\!_{\infty}(\mathbb{Z}^n_{\operatorname{lex}})$ by $H(\mathbb{I})$.

For a topological space $X$, a family $\{A_s\mid s\in\mathscr{A}\}$
of subsets of $X$ is called \emph{locally finite} if for every point
$x\in X$ there exists an open neighbourhood $U$ of $x$ in $X$ such
that the set $\{s\in\mathscr{A}\mid U\cap A_s\}$ is finite. A subset
$A$ of $X$ is said to be
\begin{itemize}
  \item \emph{co-dense} on $X$ if $X\setminus A$ is dense in $X$;
  \item an \emph{$F_\sigma$-set} in $X$ if $A$ is a union of a countable family of closed subsets in $X$.
\end{itemize}

We recall that a topological space $X$ is said to be
\begin{itemize}
  \item \emph{compact} if each open cover of $X$ has a finite
   subcover;
  \item \emph{countably compact} if each open countable cover of
   $X$ has a finite subcover;
  \item \emph{pseudocompact} if each locally finite open cover of
  $X$ is finite;
  \item a \emph{Baire space} if for each
sequence $A_1, A_2,\ldots, A_i,\ldots$ of nowhere dense subsets of $X$ the union $\bigcup_{i=1}^\infty A_i$ is a co-dense subset of $X$;
  \item  \emph{\v{C}ech
   complete} if $X$ is Tychonoff and for every compactification
   $cX$ of $X$ the remainder $cX\setminus X$ is an $F_\sigma$-set
   in $cX$;
  \item  \emph{locally compact} if every point of $X$ has an open
   neighbourhood with the compact closure.
\end{itemize}
According to Theorem~3.10.22 of \cite{Engelking1989}, a Tychonoff
topological space $X$ is pseudocompact if and only if each
continuous real-valued function on $X$ is bounded.

It is well known that topological algebra studies the influence of
topological properties of its objects on their algebraic properties
and the influence of algebraic properties of its objects on their
topological properties. There are two main problems in topological
algebra: the problem of non-discrete topologization and the problem
of embedding into objects with some topological-algebraic
properties.

In mathematical literature the question about non-discrete
(Hausdorff) topologization was posed by Markov \cite{Markov1945}.
Pontryagin gave well known conditions a base at the unity of a group
for its non-discrete topologization (see Theorem~4.5 of
\cite{HewittRoos1963} or Theorem~3.9 of \cite{Pontryagin1966}).
Various authors have refined Markov's question: can a given infinite
group $G$ endowed with a non-discrete group topology be embedded
into a compact topological group? Again, for an arbitrary Abelian
group $G$ the answer is affirmative, but there is a non-Abelian
topological group that cannot be embedded into any compact
topological group ({see Section~9 of \cite{HBSTT}}).

Also, Ol'shanskiy \cite{Olshansky1980} constructed an infinite
countable group $G$ such that every Hausdorff group topology on $G$
is discrete. Eberhart and Selden showed in \cite{EberhartSelden1969}
that every Hausdorff semigroup topology on the bicyclic semigroup
$\mathscr{C}(p,q)$ is discrete. Bertman and West proved in
\cite{BertmanWest1976} that every Hausdorff topology $\tau$ on
$\mathscr{C}(p,q)$ such that $(\mathscr{C}(p,q),\tau)$ is a
semitopological semigroup is also discrete. Taimanov gave in
\cite{Taimanov1975} sufficient conditions on a commutative semigroup
to have a non-discrete semigroup topology.

Many mathematiciants have studied the problems of embeddings of
topological semigroups into compact or compact-like topological
semigroups (see \cite{CHK}). Neither stable nor $\Gamma$-compact
topological semigroups can contain a copy of the bicyclic
semigroup~\cite{AHK, HildebrantKoch1988}. Also, the bicyclic
semigroup cannot be embedded into any countably compact topological
inverse semigroup~\cite{GutikRepovs2007}. Moreover, the conditions
were given in \cite{BanakhDimitrovaGutik2009} and
\cite{BanakhDimitrovaGutik2010} when a countably compact or
pseudocompact topological semigroup cannot contain the bicyclic
semigroup.

However, Banakh, Dimitrova and Gutik~\cite{BanakhDimitrovaGutik2010}
have constructed (assuming the Continuum Hypothesis or Martin
Axiom) an example of a Tychonoff countably compact topological
semigroup which contains the bicyclic semigroup. The problems of
topologization of semigroups of partial transformations and their
embeddings into compact-like semigroup were studied in
\cite{GutikPavlyk2005, GutikPavlykReiter2009, GutikReiter2009, GutikReiter2010}.

Doroshenko in \cite{Doroshenko2005, Doroshenko2009} studied the semigroups of endomorphisms of linearly ordered sets $\mathbb{N}$ and $\mathbb{Z}$ and their subsemigroups of cofinite endomorphisms. In \cite{Doroshenko2009} he described the Green relations, groups of automorphisms, conjugacy, centralizers of elements, growth, and free subsemigroups in these subgroups. In \cite{Doroshenko2005} there was shown  that  both these semigroups do not admit an irreducible system of generators. In their  subsemigroups of cofinite functions all irreducible systems of generators are
described there. Also, here the last semigroups are presented in terms of generators and relations.

Gutik and Repov\v{s} in \cite{GutikRepovs2011} showed that the semigroup
$\mathscr{I}_{\infty}^{\!\nearrow}(\mathbb{N})$ of partial cofinite
monotone injective transformations of the set of positive integers
$\mathbb{N}$ has algebraic properties similar to those of the
bicyclic semigroup: it is bisimple and all of its non-trivial
semigroup homomorphisms are either isomorphisms or group
homomorphisms. There were proved that every locally compact topology $\tau$
on $\mathscr{I}_{\infty}^{\!\nearrow}(\mathbb{N})$ such that
$(\mathscr{I}_{\infty}^{\!\nearrow}(\mathbb{N}),\tau)$ is a
topological inverse semigroup, is discrete and the
closure of $(\mathscr{I}_{\infty}^{\!\nearrow}(\mathbb{N}),\tau)$ in
a topological semigroup was described.

In \cite{GutikRepovs2012} Gutik and Repov\v{s} studied the semigroup
$\mathscr{I}_{\infty}^{\!\nearrow}(\mathbb{Z})$ of partial cofinite
monotone injective transformations of the set of integers
$\mathbb{Z}$ and they showed that
$\mathscr{I}^{\!\nearrow}_{\infty}(\mathbb{Z})$ is bisimple and all
of its non-trivial semigroup homomorphisms are either isomorphisms
or group homomorphisms. Also they proved that every Baire
topology $\tau$ on $\mathscr{I}^{\!\nearrow}_{\infty}(\mathbb{Z})$
such that $(\mathscr{I}^{\!\nearrow}_{\infty}(\mathbb{Z}),\tau)$ is
a Hausdorff semitopological semigroup is discrete and construct a
non-discrete Hausdorff semigroup inverse topology $\tau_W$ on
$\mathscr{I}^{\!\nearrow}_{\infty}(\mathbb{Z})$.

In this paper we study the semigroup
$\mathscr{I\!O}\!_{\infty}(\mathbb{Z}^n_{\operatorname{lex}})$. We describe
Green's relations on $\mathscr{I\!O}\!_{\infty}(\mathbb{Z}^n_{\operatorname{lex}})$,
show that the semigroup $\mathscr{I\!O}\!_{\infty}(\mathbb{Z}^n_{\operatorname{lex}})$ is
bisimple and establish its projective congruences. We prove that $\mathscr{I\!O}\!_{\infty}(\mathbb{Z}^n_{\operatorname{lex}})$ is finitely generated, every automorphism of $\mathscr{I\!O}\!_{\infty}(\mathbb{Z})$ is inner and show that  in the case $n\geqslant 2$ the semigroup $\mathscr{I\!O}\!_{\infty}(\mathbb{Z}^n_{\operatorname{lex}})$ has non-inner automorphisms. Also we prove that every Baire topology $\tau$ on $\mathscr{I\!O}\!_{\infty}(\mathbb{Z}^n_{\operatorname{lex}})$
such that $(\mathscr{I\!O}\!_{\infty}(\mathbb{Z}^n_{\operatorname{lex}}),\tau)$ is
a Hausdorff semitopological semigroup is discrete and construct a non-discrete Hausdorff semigroup inverse topology on $\mathscr{I\!O}\!_{\infty}(\mathbb{Z}^n_{\operatorname{lex}})$. We show that the discrete semigroup
$\mathscr{I\!O}\!_{\infty}(\mathbb{Z}^n_{\operatorname{lex}})$ cannot be embedded
into some classes of compact-like topological semigroups and that its remainder under the closure in a topological semigroup $S$ is an ideal in $S$.


\section{Algebraic properties of the semigroup
$\mathscr{I\!O}\!_{\infty}(\mathbb{Z}^n_{\operatorname{lex}})$}\label{section-2}

\begin{lemma}\label{lemma-2.1}
Let $n$ be any positive integer $\ge 2$, $\alpha\in\mathscr{I\!O}\!_{\infty}(\mathbb{Z}^n_{\operatorname{lex}})$ and $(i,l)\alpha=(j,m)$. Then $i=j$.
\end{lemma}

\begin{proof}
We shall show the assertion of the lemma by induction. Let $i=1$. Suppose the contrary: there exists an integer $l$ such that $(1,l)\alpha=(j,m)$ and $j\ge 2$. Then the injectivity and monotonicity of $\alpha$ imply that $(1,k)\alpha\geqslant(j,m)$ for every integer $k\ge l$. This contradicts the cofinality of $\alpha$, and hence we get $j=1$.

Next we shall prove that if the assertion of the lemma is true for all positive integers $i<p$, where $p\le n$, then it is true for $i=p$. Suppose to the contrary that there exists an integer $l$ such that $(p,l)\alpha=(j,m)$ and $j>p$. Then the injectivity and monotonicity of $\alpha$ imply that $(p,k)\alpha\geqslant(j,m)$ for every integer $k\ge l$. By assumption of induction we get that the set $(L_n\times\mathbb{Z})\setminus\operatorname{ran}\alpha$ is infinite, which contradicts the cofinality of $\alpha$. The obtained contradiction implies the equality $j=p$. This completes the proof of the lemma.
\end{proof}

\begin{proposition}\label{proposition-2.2}
Let $n$ be any positive integer $\ge 2$. Then every two cofinite subset of $L_n\times_{\operatorname{lex}}\mathbb{Z}$ are order isomorphic.
\end{proposition}

\begin{proof}
The statement of the lemma is trivial in the case when $n=1$. Let $A$ and $B$ are cofinite subset of $L_n\times_{\operatorname{lex}}\mathbb{Z}$. Then for every $i=1,\ldots,n$, the sets $A\cap (\{i\}\times\mathbb{Z})$ and $B\cap (\{i\}\times\mathbb{Z})$ are cofinite subsets of $\{i\}\times\mathbb{Z}$, and hence are order isomorphic. This implies that the union of their coordinatewise order isomorphisms on the first factor is an order isomorphism of $A$ and $B$.
\end{proof}

For every $i=1,\ldots,n$ we put
\begin{equation*}
    S_i=\left\{\alpha\in\mathscr{I\!O}\!_{\infty}(\mathbb{Z}^n_{\operatorname{lex}}) \colon \hbox{the restriction~} \alpha|_{(L_n\setminus\{i\})\times\mathbb{Z}} \hbox{~is an identity map}\right\}.
\end{equation*}
It is obvious that $S_i$ is an inverse submonoid of the semigroup $\mathscr{I\!O}\!_{\infty}(\mathbb{Z}^n_{\operatorname{lex}})$ for every $i=1,\ldots,n$.

\begin{proposition}\label{proposition-2.3}
Let $n$ be any positive integer $\ge 2$.  Then the following assertions hold:
\begin{itemize}
  \item[$(i)$] for every $i=1,\ldots,n$ the semigroup $S_i$ is isomorphic to $\mathscr{I\!O}\!_{\infty}(\mathbb{Z})$;
  \item[$(ii)$] $S_i\cap S_j=\{\mathbb{I}\}$ for all distinct $i,j=1,\ldots,n$;
  \item[$(iii)$] if $i\neq j$, $i,j=1,\ldots,n$, then $\alpha_i\beta_j=\beta_j\alpha_i$ for all $\alpha_i\in S_i$ and $\beta_j\in S_j$;
  \item[$(iv)$] the semigroup $\mathscr{I\!O}\!_{\infty}(\mathbb{Z}^n_{\operatorname{lex}})$ is isomorphic to the direct product $\prod_{i=1}^{n}S_i$, and hence it is isomorphic to the direct power $\left(\mathscr{I\!O}\!_{\infty}(\mathbb{Z})\right)^{n}$.
\end{itemize}
\end{proposition}

\begin{proof}
$(i)$ For fixed $i=1,\ldots,n$ we identify the semigroups $S_i$ and $\mathscr{I\!O}\!_{\infty}(\mathbb{Z})$ by the map $\mathbf{F}_i\colon S_i\to \mathscr{I\!O}\!_{\infty}(\mathbb{Z})$, where $(\alpha)\mathbf{F}_i= \alpha|_{\{i\}\times\mathbb{Z}}$ is the restriction of $\alpha$ onto $\{i\}\times\mathbb{Z}$. Simple verifications show that such defined map $\mathbf{F}_i\colon S_i\to \mathscr{I\!O}\!_{\infty}(\mathbb{Z})$ is a semigroup isomorphism.

Statements $(ii)$ and $(iii)$ are trivial and follow from the definition of the semigroup $S_i$, $i=1,\ldots,n$.

$(iv)$ We define the map $\mathbf{I}\colon \mathscr{I\!O}\!_{\infty}(\mathbb{Z}^n_{\operatorname{lex}}) \to \prod_{i=1}^{n}S_i\colon \alpha\mapsto(\alpha_1,\ldots,\alpha_n)$, where \begin{equation*}
    (x)\alpha_i=
    \left\{
      \begin{array}{cl}
        (x)\alpha, & \hbox{if~} x\in\{i\}\times\mathbb{Z};\\
        x, & \hbox{otherwise,}
      \end{array}
    \right.
\end{equation*}
$i=1,\ldots,n$. Simple verifications imply that the map $\mathbf{I}_i\colon \mathscr{I\!O}\!_{\infty}(\mathbb{Z}^n_{\operatorname{lex}})\to S_i$, defined by the formula $(\alpha)\mathbf{I}_i=\alpha_i$ is a homomorphism. This implies that the map $\mathbf{I}\colon \mathscr{I\!O}\!_{\infty}(\mathbb{Z}^n_{\operatorname{lex}}) \to \prod_{i=1}^{n}S_i$ is a homomorphism. Also, for arbitrary  $\alpha_1\in S,\ldots,\alpha_n\in S_n$ we have that $(\alpha)\mathbf{I}=(\alpha_1,\ldots,\alpha_n)$, where $\alpha=\alpha_1\ldots\alpha_n$, and hence the map $\mathbf{I}$ is surjective. If $\alpha$ and $\beta$ are distinct elements of the semigroup $\mathscr{I\!O}\!_{\infty}(\mathbb{Z}^n_{\operatorname{lex}})$, then there exists a positive integer $i\in\{1,\ldots,n\}$ such that $(x)\alpha\neq(x)\beta$ for some $x\in\{i\}\times\mathbb{Z}$, and hence we have that $(x)\alpha_i\neq(x)\beta_i$. This implies that $(\alpha)\mathbf{I}\neq(\beta)\mathbf{I}$, and hence the map $\mathbf{I}\colon \mathscr{I\!O}\!_{\infty}(\mathbb{Z}^n_{\operatorname{lex}}) \to \prod_{i=1}^{n}S_i$ is an isomorphism. The last statement follows from $(i)$.
\end{proof}

\begin{proposition}\label{proposition-2.4}
Let $n$ be any positive integer. Then the following assertions hold:
\begin{itemize}
    \item[$(i)$] An element $\alpha$ of the semigroup
         $\mathscr{I\!O}\!_{\infty}(\mathbb{Z}^n_{\operatorname{lex}})$
         is an idempotent if and only if $(x)\alpha=x$ for every
         $x\in\operatorname{dom}\alpha$.

    \item[$(ii)$] If $\varepsilon,\iota\in
          E(\mathscr{I\!O}\!_{\infty}(\mathbb{Z}^n_{\operatorname{lex}}))$,
          then $\varepsilon\leqslant\iota$ if and only if
          $\operatorname{dom}\varepsilon\subseteq
          \operatorname{dom}\iota$.

    \item[$(iii)$] The semilattice
          $E(\mathscr{I\!O}\!_{\infty}(\mathbb{Z}^n_{\operatorname{lex}}))$ is
          isomorphic to $(\mathscr{P}_{<\omega}(L_n\times\mathbb{Z}),\subseteq)$ under the mapping $(\varepsilon)h=\left( L_n\times_{\operatorname{lex}} \mathbb{Z}\right)\setminus \operatorname{dom}\varepsilon$.

    \item[$(iv)$] Every maximal chain in
          $E(\mathscr{I\!O}\!_{\infty}(\mathbb{Z}^n_{\operatorname{lex}}))$ is an          $\omega$-chain.

    \item[$(v)$] $\alpha\mathscr{R}\beta$ in
         $\mathscr{I\!O}\!_{\infty}(\mathbb{Z}^n_{\operatorname{lex}})$ if and only if $\operatorname{dom}\alpha=\operatorname{dom}\beta$.

    \item[$(vi)$] $\alpha\mathscr{L}\beta$ in
         $\mathscr{I\!O}\!_{\infty}(\mathbb{Z}^n_{\operatorname{lex}})$ if and only if $\operatorname{ran}\alpha=\operatorname{ran}\beta$.

    \item[$(vii)$] $\alpha\mathscr{H}\beta$ in
         $\mathscr{I\!O}\!_{\infty}(\mathbb{Z}^n_{\operatorname{lex}})$ if and only if $\operatorname{dom}\alpha=\operatorname{dom}\beta$ and
         $\operatorname{ran}\alpha=\operatorname{ran}\beta$.

    \item[$(viii)$] For all idempotents $\varepsilon,\varphi\in
         \mathscr{I\!O}\!_{\infty}(\mathbb{Z}^n_{\operatorname{lex}})$ there exist infinitely many elements $\alpha,\beta\in
         \mathscr{I\!O}\!_{\infty}(\mathbb{Z}^n_{\operatorname{lex}})$ such that
         $\alpha\cdot\beta=\varepsilon$ and
         $\beta\cdot\alpha=\varphi$.

     \item[$(ix)$] $\mathscr{I\!O}\!_{\infty}(\mathbb{Z}^n_{\operatorname{lex}})$ is a bisimple semigroup and hence
         $\mathscr{J}=\mathscr{D}$.
\end{itemize}
\end{proposition}

\begin{proof}
The proofs of assertions $(i)$--$(iv)$ are trivial and they follow from the definition of the semigroup          $\mathscr{I\!O}\!_{\infty}(\mathbb{Z}^n_{\operatorname{lex}})$.

The proofs of $(v)$-–$(vii)$ follow trivially from the fact that $\mathscr{I\!O}\!_{\infty}(\mathbb{Z}^n_{\operatorname{lex}})$ is a regular
semigroup, and by \cite[Proposition 2.4.2, Exercise 5.11.2]{Howie1995}.

Proposition~\ref{proposition-2.2} implies assertion $(viii)$. Assertion $(ix)$ follows from $(viii)$ and Proposition~3.2.5$(1)$ of \cite{Lawson1998}.
\end{proof}

By Lemma~\ref{lemma-2.1}, for every $\alpha\in\mathscr{I\!O}\!_{\infty}(\mathbb{Z}^n_{\operatorname{lex}})$ and any $(i,k)\in\operatorname{dom}\alpha\subseteq L_n\times\mathbb{Z}$ there exists an integer $(k)\alpha^i$ such that $(i,k)\alpha=(i,(k)\alpha^i)$. This implies that the notion $(k)\alpha^i$ well-defined for every $\alpha\in\mathscr{I\!O}\!_{\infty}(\mathbb{Z}^n_{\operatorname{lex}})$ and any $(i,k)\in\operatorname{dom}\alpha$. Also, later we shall identify $\alpha^i$ with the restriction $\alpha|_{\{i\}\times\mathbb{Z}}$ of $\alpha$ on the set $\{i\}\times\mathbb{Z}$. This makes to possible to consider $\alpha^i$ as an element of the semigroup $\mathscr{I\!O}\!_{\infty}(\mathbb{Z})$.

\begin{lemma}\label{lemma-2.5}
Let $n$ be any positive integer. Then a partial injective monotone selfmap $\alpha$ of $L_n\times_{\operatorname{lex}}\mathbb{Z}$ is an element of the semigroup
$\mathscr{I\!O}\!_{\infty}(\mathbb{Z}^n_{\operatorname{lex}})$ if and only if there exist integers $d_\alpha$ and $u_\alpha$ such that for any $i=1,\ldots,n$ the following conditions hold:
\begin{equation*}
    (i,k-1)\alpha=(i,(k-1)\alpha^i)=(i,(k)\alpha^i-1) \quad \mbox{and} \quad
    (i,l+1)\alpha=(i,(l+1)\alpha^i)=(i,(l)\alpha^i+1),
\end{equation*}
for all integers $k\leqslant d_\alpha$ and $l\geqslant u_\alpha$. Moreover $\alpha\in H(\mathbb{I})$ in $\mathscr{I\!O}\!_{\infty}(\mathbb{Z}^n_{\operatorname{lex}})$ if and only if
\begin{equation*}
(i,m+1)\alpha=(i,(m+1)\alpha^i)=(i,(m)\alpha^i+1),
\end{equation*}
for any $i=1,\ldots,n$ and any integer $m$.
\end{lemma}

\begin{proof}
By Lemma~1.1 from \cite{GutikRepovs2012} we have that a partial injective monotone selfmap $\alpha$ of $\mathbb{Z}$ is an element of the semigroup
$\mathscr{I\!O}\!_{\infty}(\mathbb{Z})$ if and only if there exist integers $d_\alpha$ and $u_\alpha$ such that the following conditions hold:
\begin{equation*}
    (k-1)\alpha=(k)\alpha-1 \quad \mbox{and} \quad
    (l+1)\alpha=(l)\alpha+1 \quad \mbox{for all integers }
    k\leqslant d_\alpha \mbox{ and } l\geqslant u_\alpha,
\end{equation*}
and $\alpha\in H(\mathbb{I})$ in $\mathscr{I\!O}\!_{\infty}(\mathbb{Z})$ if and only if $(m+1)\alpha=(m)\alpha+1$ for any integer $m$. Then Proposition~\ref{proposition-2.3} implies that a partial injective monotone selfmap $\alpha$ of $L_n\times_{\operatorname{lex}}\mathbb{Z}$ is an element of the semigroup
$\mathscr{I\!O}\!_{\infty}(\mathbb{Z}^n_{\operatorname{lex}})$ if and only if for every $i=1,\ldots,n$ there exist integers $d_\alpha^i$ and $u_\alpha^i$ such that
\begin{equation*}
    (i,k-1)\alpha=(i,(k-1)\alpha^i)=(i,(k)\alpha^i-1) \quad \mbox{and} \quad
    (i,l+1)\alpha=(i,(l+1)\alpha^i)=(i,(l)\alpha^i+1),
\end{equation*}
for all integers $k\leqslant d_\alpha^i$ and $l\geqslant u_\alpha^i$. We put $d_\alpha=\min\{d_\alpha^1,\ldots,d_\alpha^n\}$ and $u_\alpha=\max\{u_\alpha^1,\ldots,u_\alpha^n\}$. Simple verifications show that the integers $d_\alpha$ and $u_\alpha$ are requested.

The last statement immediately follows from Proposition~\ref{proposition-2.3} and Lemma~1.1 of~\cite{GutikRepovs2012}.
\end{proof}

The second part of Lemma~\ref{lemma-2.5} implies the following proposition:

\begin{proposition}\label{proposition-2.6}
For any positive integer $n$ the group of units $H(\mathbb{I})$ of the semigroup $\mathscr{I\!O}\!_{\infty}(\mathbb{Z}^n_{\operatorname{lex}})$ is isomorphic to the direct power $\left(\mathbb{Z}(+)\right)^n$.
\end{proposition}

Theorem~2.20 of~\cite{CP}, Proposition~\ref{proposition-2.4}$(ix)$ and Proposition~\ref{proposition-2.6} imply the following corollary:

\begin{corollary}\label{corollary-2.7}
Let $n$ be any positive integer. Then every maximal subgroup of the semigroup
$\mathscr{I\!O}\!_{\infty}(\mathbb{Z}^n_{\operatorname{lex}})$ is isomorphic to
the direct power $\left(\mathbb{Z}(+)\right)^n$.
\end{corollary}

\begin{proposition}\label{proposition-2.8}
Let $n$ be any positive integer. Then for every elements
$\alpha,\beta\in\mathscr{I\!O}\!_{\infty}(\mathbb{Z}^n_{\operatorname{lex}})$, both
sets
 $
\left\{\chi\in\mathscr{I\!O}\!_{\infty}(\mathbb{Z}^n_{\operatorname{lex}}) \colon
\alpha\cdot\chi=\beta\right\}
 $
and
 $
\{\chi\in\mathscr{I\!O}\!_{\infty}(\mathbb{Z}^n_{\operatorname{lex}}) \colon
\chi\cdot\alpha=\beta\}
 $
are finite.
\end{proposition}

\begin{proof}
We denote
$A=\{\chi\in\mathscr{I\!O}\!_{\infty}(\mathbb{Z}^n_{\operatorname{lex}}) \colon
\alpha\cdot\chi=\beta\}$ and
$B=\{\chi\in\mathscr{I\!O}\!_{\infty}(\mathbb{Z}^n_{\operatorname{lex}}) \colon
\alpha^{-1}\cdot\alpha\cdot\chi=\alpha^{-1}\cdot\beta\}$. Then
$A\subseteq B$ and the restriction of any partial map $\chi\in B$ to
$\operatorname{dom}(\alpha^{-1}\cdot\alpha)$ coincides with the
partial map $\alpha^{-1}\cdot\beta$. Since every partial map from
$\mathscr{I\!O}\!_{\infty}(\mathbb{Z}^n_{\operatorname{lex}})$ is monotone we
conclude that the set $B$ is finite and hence so is $A$. The proof of the other case is similar.
\end{proof}

The following theorem describes the least group congruence $\sigma$ on the semigroup $\mathscr{I\!O}\!_{\infty}(\mathbb{Z}^n_{\operatorname{lex}})$.

\begin{theorem}\label{theorem-2.9}
Let $n$ be any positive integer. Then the quotient semigroup
$\mathscr{I\!O}\!_{\infty}(\mathbb{Z}^n_{\operatorname{lex}})/\sigma$ is isomorphic to the direct power $\left(\mathbb{Z}(+)\right)^{2n}$.
\end{theorem}

\begin{proof}
Let $\alpha$ and $\beta$ be $\sigma$-equivalent elements of the semigroup $\mathscr{I\!O}\!_{\infty}(\mathbb{Z}^n_{\operatorname{lex}})$. Then by Lemma~III.5.2 from \cite{Petrich1984} there exists an idempotent $\varepsilon_0$ in
$\mathscr{I\!O}\!_{\infty}(\mathbb{Z}^n_{\operatorname{lex}})$ such that
$\alpha\cdot\varepsilon_0=\beta\cdot\varepsilon_0$. Since
$\mathscr{I\!O}\!_{\infty}(\mathbb{Z}^n_{\operatorname{lex}})$ is an inverse
semigroup we conclude that $\alpha\cdot\varepsilon=\beta\cdot\varepsilon$ for all
$\varepsilon\in E(\mathscr{I\!O}\!_{\infty}(\mathbb{Z}^n_{\operatorname{lex}}))$
such that $\varepsilon\leqslant\varepsilon_0$. Then Lemma~\ref{lemma-2.5} implies that there exist integers $d_{\alpha}$, $u_{\alpha}$, $d_{\beta}$ and $u_{\beta}$ such that for any $i=1,\ldots,n$ the following conditions hold:
\begin{align*}
    &(i,k-1)\alpha=(i,(k-1)\alpha^i)=(i,(k)\alpha^i-1),
    &(i,l+1)\alpha=(i,(l+1)\alpha^i)=(i,(l)\alpha^i+1),\\
    &(i,k-1)\beta=(i,(k-1)\beta^i)=(i,(k)\beta^i-1),  &(i,l+1)\beta=(i,(l+1)\beta^i)=(i,(l)\beta^i+1),
\end{align*}
for all integers $k\leqslant d=\min\{d_\alpha,d_\beta\}$ and $l\geqslant u =\max\{u_\alpha,u_\beta\}$. We put
\begin{equation*}
d_0=\min\left\{(d)\alpha^1,\ldots,(d)\alpha^n,(d)\beta^1,\ldots,(d)\beta^n\right\} \quad \hbox{and}\quad u_0=\max\left\{(u)\alpha^1,\ldots,(u)\alpha^n,(u)\beta^1,\ldots,(u)\beta^n\right\}.
\end{equation*}
Let $\varepsilon_1$ be an identity map from
$L_n\times\left(\mathbb{Z}\setminus\{d_0,d_0+1,\ldots,u_0\}\right)$ onto itself. Then $\varepsilon^0=\varepsilon_1\cdot\varepsilon_0\leqslant \varepsilon_0$ and hence we have that $\alpha\cdot\varepsilon^0=\beta\cdot\varepsilon^0$. Therefore we have showed that if the elements $\alpha$ and $\beta$ of the semigroup $\mathscr{I\!O}\!_{\infty}(\mathbb{Z}^n_{\operatorname{lex}})$ are $\sigma$-equivalent, then there exist integers $d$ and $u$ such that
\begin{equation*}
(i,k)\alpha=(i,k)\beta  \qquad \mbox{ and } \qquad (i,l)\alpha=(i,l)\beta,
\end{equation*}
for all integers $k\leqslant d$ and $l\geqslant u$ and any $i=1,\ldots,n$.

Conversely, suppose that exist integers $d$ and $u$ such that
\begin{equation*}
(i,k)\alpha=(i,k)\beta  \qquad \mbox{ and } \qquad (i,l)\alpha=(i,l)\beta,
\end{equation*}
for all integers $k\leqslant d$ and $l\geqslant u$ and any $i=1,\ldots,n$. Then we have that $d\leqslant u$. If $d=u$ or $d=u-1$ then $\alpha=\beta$ in
$\mathscr{I\!O}\!_{\infty}(\mathbb{Z}^n_{\operatorname{lex}})$ and hence $\alpha$
and $\beta$ are $\sigma$-equivalent. If $d<u-1$ then we put $\varepsilon_0$ to be the identity map of the set
\begin{equation*}
\left(L_n\times\mathbb{Z}\right)\setminus \left\{\left(1,(d+1)\alpha^1\right),\ldots,\left(1,(u-1)\alpha^1\right),\ldots, \left(n,(d+1)\alpha^n\right),\ldots,\left(n,(u-1)\alpha^n\right) \right\}.
\end{equation*}
Then we get that $(i,k)(\alpha\circ\varepsilon_0)=(i,k)(\beta\circ\varepsilon_0)$
for any $(i,k)\in L_n\times\left(\mathbb{Z}\setminus\{d+1,\ldots,u-1\}\right)$ and therefore $\alpha\cdot\varepsilon_0=\beta\cdot\varepsilon_0$. Hence Lemma~III.5.2 from \cite{Petrich1984} implies that $\alpha$ and $\beta$ are $\sigma$-equivalent elements of the semigroup
$\mathscr{I\!O}\!_{\infty}(\mathbb{Z}^n_{\operatorname{lex}})$.

Now we define the map
$\textbf{H}\colon\mathscr{I\!O}\!_{\infty}(\mathbb{Z}^n_{\operatorname{lex}}) \rightarrow \left(\mathbb{Z}(+)\times\mathbb{Z}(+)\right)^n$ by the formula
\begin{equation*}
    (\alpha)\textbf{H}=\left(\left((d_\alpha)\alpha^1-d_\alpha,
    (u_\alpha)\alpha^1-u_\alpha\right),\ldots, \left((d_\alpha)\alpha^n-d_\alpha,
    (u_\alpha)\alpha^n-u_\alpha\right)\right),
\end{equation*}
where the integers $d_\alpha$ and $u_\alpha$ are defined in
Lemma~\ref{lemma-2.5}.

We observe that
\begin{equation*}
    (d_\alpha-k)\alpha^i=(d_\alpha)\alpha^i-k \qquad \hbox{ and } \qquad
    (u_\alpha+k)\alpha^i=(u_\alpha)\alpha^i+k,
\end{equation*}
for any $i=1,\ldots,n$ and any positive integer $k$. Hence we have that
\begin{equation*}
    (k)\alpha^i-k=(d_\alpha)\alpha^i-d_\alpha \qquad \hbox{ and } \qquad
    (l)\alpha^i-l=(u_\alpha)\alpha^i-u_\alpha,
\end{equation*}
for any $i=1,\ldots,n$ and all integers $k\leqslant d_\alpha$ and $l\geqslant u_\alpha$.

Lemma~\ref{lemma-2.5} implies that there exist integers $d^0$ and
$u^0$ such that
\begin{align*}
    (k-1)\alpha^i&=(k)\alpha^i-1, &(l+1)\alpha^i&=(l)\alpha^i+1,\\
    (k-1)\beta^i&=(k)\beta^i-1,   &(l+1)\beta^i&=(l)\beta^i+1,\\
    (k-1)(\alpha^i\cdot\beta^i)&=(k)(\alpha^i\cdot\beta^i)-1,
         &(l+1)(\alpha^i\cdot\beta^i)&=(l)(\alpha^i\cdot\beta^i)+1,
\end{align*}
for any $i=1,\ldots,n$ and all integers $k\leqslant d^0$ and $l\geqslant u^0$. Hence for any $i=1,\ldots,n$ and all integers $k\leqslant d^0$ and $l\geqslant u^0$ we have that
\begin{equation*}
    (k)(\alpha^i\cdot\beta^i)-k=
    (k)(\alpha^i\cdot\beta^i)-(k)\alpha^i+(k)\alpha^i-k=
    ((d_\beta)\beta^i-d_\beta)+((d_\alpha)\alpha^i-d_\alpha),
\end{equation*}
\begin{equation*}
    (l)(\alpha^i\cdot\beta^i)-l=
    (l)(\alpha^i\cdot\beta^i)-(l)\alpha^i+(l)\alpha^i-l=
    ((u_\beta)\beta^i-u_\beta)+((u_\alpha)\alpha^i-u_\alpha).
\end{equation*}
This implies that the map
$\textbf{H}\colon\mathscr{I\!O}\!_{\infty}(\mathbb{Z}^n_{\operatorname{lex}}) \rightarrow \left(\mathbb{Z}(+)\times\mathbb{Z}(+)\right)^n$ is a homomorphism. Simple verifications show that the map $\textbf{H}$ is surjective and $\ker\textbf{H}=\sigma$, i.e., the homomorphism $\textbf{H}$ generated the congruence $\sigma$ on the semigroup $\mathscr{I\!O}\!_{\infty}(\mathbb{Z}^n_{\operatorname{lex}})$.
\end{proof}

Next we establish congruences on the semigroup $\mathscr{I\!O}\!_{\infty}(\mathbb{Z}^n_{\operatorname{lex}})$.

By Proposition~\ref{proposition-2.3}$(iv)$, the semigroup $\mathscr{I\!O}\!_{\infty}(\mathbb{Z}^n_{\operatorname{lex}})$ is isomorphic to the direct power $\left(\mathscr{I\!O}\!_{\infty}(\mathbb{Z})\right)^{n}$. Hence every element $\alpha$ of $\mathscr{I\!O}\!_{\infty}(\mathbb{Z}^n_{\operatorname{lex}})$ we can present in the form $(\alpha_1,\alpha_2,\ldots,\alpha_n)$. Later by $\alpha_i^\circ$ we shall denote the element of the form $(\mathbb{I}_1,\ldots,\mathbb{I}_{i-1},\alpha_i,\mathbb{I}_{i+1},\ldots,\mathbb{I}_n)$, where $\mathbb{I}_{j}$ is the identity of $j$-th factor in $\left(\mathscr{I\!O}\!_{\infty}(\mathbb{Z})\right)^{n}$.

For $i=1,\ldots,n$ we define a relation $\sigma_{[i]}$ on $\mathscr{I\!O}\!_{\infty}(\mathbb{Z}^n_{\operatorname{lex}})$ in the following way:
\begin{equation*}
    \alpha\sigma_{[i]}\beta \quad \hbox{ if and only if there exists an idempotent } \varepsilon\in\mathscr{I\!O}\!_{\infty}(\mathbb{Z}) \hbox{ such that } \alpha\varepsilon_i^\circ=\beta\varepsilon_i^\circ.
\end{equation*}

\begin{remark}\label{remark-2.10}
For every $\alpha=(\alpha_1,\ldots,\alpha_n)\in \mathscr{I\!O}\!_{\infty}(\mathbb{Z}^n_{\operatorname{lex}})$ we have that $\alpha=\alpha_1^\circ\ldots\alpha_n^\circ$.
\end{remark}

\begin{proposition}\label{proposition-2.11}
$\sigma_{[i]}$ is a congruence on $\mathscr{I\!O}\!_{\infty}(\mathbb{Z}^n_{\operatorname{lex}})$ for every $i=1,\ldots,n$.
\end{proposition}

\begin{proof}
It is obvious that $\sigma_{[i]}$ is reflexive and symmetric relation on $\mathscr{I\!O}\!_{\infty}(\mathbb{Z}^n_{\operatorname{lex}})$. Suppose that $\alpha\sigma_{[i]}\beta$ and $\beta\sigma_{[i]}\gamma$. Then there exist idempotents $\varepsilon,\iota\in\mathscr{I\!O}\!_{\infty}(\mathbb{Z})$ such that $\alpha\varepsilon_i^\circ=\beta\varepsilon_i^\circ$ and $\beta\iota_i^\circ=\gamma\iota_i^\circ$. Since in an inverse semigroup idempotents commute we get that $\alpha\varepsilon_i^\circ\iota_i^\circ= \beta\varepsilon_i^\circ\iota_i^\circ= \beta\iota_i^\circ\varepsilon_i^\circ=
\gamma\iota_i^\circ\varepsilon_i^\circ=\gamma\varepsilon_i^\circ\iota_i^\circ$, and hence $\alpha\sigma_{[i]}\gamma$.

Suppose that $\alpha\sigma_{[i]}\beta$ for some $\alpha,\beta\in \mathscr{I\!O}\!_{\infty}(\mathbb{Z}^n_{\operatorname{lex}})$ and $\gamma$ be any element of $\mathscr{I\!O}\!_{\infty}(\mathbb{Z}^n_{\operatorname{lex}})$. Then we have that $\alpha\varepsilon_i^\circ=\beta\varepsilon_i^\circ$ for some idempotent $\varepsilon\in\mathscr{I\!O}\!_{\infty}(\mathbb{Z})$. Now we get $\gamma\alpha\varepsilon_i^\circ=\gamma\beta\varepsilon_i^\circ$ and
\begin{equation*}
\begin{split}
  \alpha\gamma(\gamma_i^{-1}\varepsilon_i\gamma_i)^\circ= & \; \alpha(\gamma_i\gamma_i^{-1}\varepsilon_i)^\circ\gamma=
    \alpha(\varepsilon_i\gamma_i\gamma_i^{-1})^\circ\gamma=
    \alpha\varepsilon_i^\circ(\gamma_i\gamma_i^{-1})^\circ\gamma=
    \beta\varepsilon_i^\circ(\gamma_i\gamma_i^{-1})^\circ\gamma=
    \beta(\varepsilon_i\gamma_i\gamma_i^{-1})^\circ\gamma= \\
    = & \; \beta(\gamma_i\gamma_i^{-1}\varepsilon_i)^\circ\gamma= \beta\gamma(\gamma_i^{-1}\varepsilon_i\gamma_i)^\circ,
\end{split}
\end{equation*}
where $\gamma_i$ is the $i$-th coordinate of $\gamma$ of the representation in $\left(\mathscr{I\!O}\!_{\infty}(\mathbb{Z})\right)^{n}$. Since $\gamma_i^{-1}\varepsilon_i\gamma_i$ is an idempotent of $\mathscr{I\!O}\!_{\infty}(\mathbb{Z})$ we have that $(\gamma\alpha)\sigma_{[i]}(\gamma\beta)$ and $(\alpha\gamma)\sigma_{[i]}(\beta\gamma)$. This completes the proof of the proposition.
\end{proof}

\begin{proposition}\label{proposition-2.12}
$\sigma_{[i]}\circ\sigma_{[j]}=\sigma_{[j]}\circ\sigma_{[i]}$ and hence $\sigma_{[i]}\circ\sigma_{[j]}=\sigma_{[i]}\vee\sigma_{[j]}$ for any $i,j=1,\ldots,n$.
\end{proposition}

\begin{proof}
Suppose that $\alpha(\sigma_{[i]}\circ\sigma_{[j]})\beta$ for some $\alpha=(\alpha_1,\ldots,\alpha_n), \beta=(\beta_1,\ldots,\beta_n)\in \mathscr{I\!O}\!_{\infty}(\mathbb{Z}^n_{\operatorname{lex}})$. Then there exist $\gamma=(\gamma_1,\ldots,\gamma_n)\in \mathscr{I\!O}\!_{\infty}(\mathbb{Z}^n_{\operatorname{lex}})$ such that $\alpha\sigma_{[i]}\gamma$ and $\gamma\sigma_{[j]}\beta$. Then the definition of $\sigma_{[i]}$ implies that the following equalities hold:
\begin{gather*}
    \alpha_k=\gamma_k, \quad \hbox{ for all }k\in\{1,\ldots,n\}\setminus\{i\};\\
    \gamma_l=\beta_l,  \quad \hbox{ for all }l\in\{1,\ldots,n\}\setminus\{j\};\\
    \alpha_i\varepsilon=\gamma_i\varepsilon,\;  \gamma_j\varepsilon=\beta_j\varepsilon\;  \hbox{ for some idempotent } \varepsilon\in \mathscr{I\!O}\!_{\infty}(\mathbb{Z}).
\end{gather*}
We put $\delta=(\delta_1,\ldots,\delta_n)$, where
\begin{equation*}
     \delta_l=
\left\{
  \begin{array}{ll}
    \beta_j, & \hbox{if~} l=j;\\
    \alpha_l, & \hbox{if~} l\neq j.
  \end{array}
\right.
\end{equation*}
Then we get that $\alpha\sigma_{[j]}\delta$ and $\delta\sigma_{[i]}\beta$, and hence $\alpha(\sigma_{[j]}\circ\sigma_{[i]})\beta$. This implies that $\sigma_{[i]}\circ\sigma_{[j]}\subseteq \sigma_{[j]}\circ\sigma_{[i]}$ and hence by Lemma~1.4 from \cite{CP} we get that $\sigma_{[i]}\circ\sigma_{[j]}=\sigma_{[i]}\vee\sigma_{[j]}$.
\end{proof}

\begin{proposition}\label{proposition-2.13}
For any collection $\{i_1,\ldots,i_k\}\subseteq\{1,\ldots,n\}$ of distinct indices, $k\le n$, the following condition holds
$\sigma_{[i_1]}\circ\ldots\circ\sigma_{[i_k]}= \sigma_{[i_1]}\vee\ldots\vee\sigma_{[i_k]}$, and hence $\sigma_{[i_1,\ldots,i_k]}= \sigma_{[i_1]}\circ\ldots\circ\sigma_{[i_k]}$ is a congruence on $\mathscr{I\!O}\!_{\infty}(\mathbb{Z}^n_{\operatorname{lex}})$.
\end{proposition}

\begin{proof}
We prove the statements of the proposition by induction. Proposition~\ref{proposition-2.12} implies that the statements hold for $k=2$. Now we suppose that the assertion holds for any integer $j<k_0\le n$ and we shall show that it is true for $k_0$. Then we have
\begin{equation*}
\begin{split}
  (\sigma_{[i_1]}\circ\ldots\circ\sigma_{[i_{k_0-1}]})\circ\sigma_{[i_{k_0}]}=
    &\;(\sigma_{[i_1]}\circ\ldots\circ\sigma_{[i_{k_0-2}]})
    \circ(\sigma_{[i_{k_0-1}]}\circ\sigma_{[i_{k_0}]})=\\
    =&\;(\sigma_{[i_1]}\circ\ldots\circ\sigma_{[i_{k_0-2}]})
    \circ(\sigma_{[i_{k_0}]}\circ\sigma_{[i_{k_0-1}]})=\\
    =&\;\sigma_{[i_1]}\circ\ldots\circ(\sigma_{[i_{k_0}]}
    \circ\sigma_{[i_{k_0-2}]})\circ\sigma_{[i_{k_0-1}]}=\\
    =&\;(\sigma_{[i_1]}\circ\ldots\circ\sigma_{[i_{k_0}]})
    \circ(\sigma_{[i_{k_0-2}]}\circ\sigma_{[i_{k_0-1}]})=\\
    =&\;\ldots=\\
    =&\;\sigma_{[i_{k_0}]}\circ(\sigma_{[i_1]}\circ \ldots\sigma_{[i_{k_0-2}]}\circ\sigma_{[i_{k_0-1}]}).
\end{split}
\end{equation*}
This implies the following
\begin{equation*}
    \sigma_{[i_1]}\circ\ldots\circ\sigma_{[i_{k_0-1}]}\circ\sigma_{[i_{k_0}]}=
    (\sigma_{[i_1]}\circ\ldots\circ\sigma_{[i_{k_0-1}]})\circ\sigma_{[i_{k_0}]}=
    (\sigma_{[i_1]}\circ\ldots\circ\sigma_{[i_{k_0-1}]})\vee\sigma_{[i_{k_0}]}=
    (\sigma_{[i_1]}\vee\ldots\vee\sigma_{[i_{k_0-1}]})\vee\sigma_{[i_{k_0}]},
\end{equation*}
and similar arguments as in the proof of Proposition~\ref{proposition-2.12} imply that $\sigma_{[i_1,\ldots,i_k]}= \sigma_{[i_1]}\circ\ldots\circ\sigma_{[i_k]}$ is a congruence on $\mathscr{I\!O}\!_{\infty}(\mathbb{Z}^n_{\operatorname{lex}})$.
\end{proof}

Proposition~\ref{proposition-2.13} implies the following

\begin{corollary}\label{corollary-2.14}
For any collections $\{i_1,\ldots,i_k\}\subseteq\{1,\ldots,n\}$ and $\{j_1,\ldots,j_l\}\subseteq\{1,\ldots,n\}$ of indices, $k\le n$, the following condition holds:
\begin{itemize}
  \item[$(i)$] $\sigma_{[i_1,\ldots,i_k]}\subseteq \sigma_{[j_1,\ldots,j_l]}$ if and only if $\{i_1,\ldots,i_k\}\subseteq\{j_1,\ldots,j_l\}$;
  \item[$(ii)$] $\sigma_{[i_1,\ldots,i_k]}=\sigma_{[j_1,\ldots,j_l]}$ if and only if $\{i_1,\ldots,i_k\}=\{j_1,\ldots,j_l\}$;
  \item[$(iii)$] $\sigma_{[i_1,\ldots,i_k]}\circ\sigma_{[j_1,\ldots,j_l]}= \sigma_{[p_1,\ldots,p_m]}$, where $\{p_1,\ldots,p_m\}= \{i_1,\ldots,i_k\}\cup\{j_1,\ldots,j_l\}$.
\end{itemize}
\end{corollary}

\begin{proposition}\label{proposition-2.15}
For any collection $\{i_1,\ldots,i_k\}\subseteq\{1,\ldots,n\}$ of distinct indices, $k\le n$, $\alpha\sigma_{[i_1,\ldots,i_k]}\beta$ in $\mathscr{I\!O}\!_{\infty}(\mathbb{Z}^n_{\operatorname{lex}})$ if and only if $\alpha\varepsilon_{i_1}^\circ\ldots\varepsilon_{i_k}^\circ= \beta\varepsilon_{i_1}^\circ\ldots\varepsilon_{i_k}^\circ$ for some idempotents $\varepsilon_{i_1}^\circ,\ldots,\varepsilon_{i_k}^\circ\in \mathscr{I\!O}\!_{\infty}(\mathbb{Z}^n_{\operatorname{lex}})$.
\end{proposition}

\begin{proof}
$(\Rightarrow)$ Suppose that $\alpha\sigma_{[i_1,\ldots,i_k]}\beta$ in $\mathscr{I\!O}\!_{\infty}(\mathbb{Z}^n_{\operatorname{lex}})$. Without loss of generality we can assume that $i_1=1,\ldots,i_k=k$. Then there exist $\gamma^1,\ldots,\gamma^{k-1}$ such that $\alpha\sigma_{[1]}\gamma^1\sigma_{[2]}\gamma^2\sigma_{[3]}\ldots \sigma_{[k-1]}\gamma^{k-1}\sigma_{[k]}\beta$. This implies the existence of idempotents $\varepsilon_1^\circ,\ldots,\varepsilon_k^\circ\in \mathscr{I\!O}\!_{\infty}(\mathbb{Z}^n_{\operatorname{lex}})$ such that $\alpha\varepsilon_1^\circ=\gamma^1\varepsilon_1^\circ$, $\gamma^1\varepsilon_2^\circ=\gamma^2\varepsilon_2^\circ$, $\ldots$, $\gamma^{k-1}\varepsilon_k^\circ=\beta\varepsilon_k^\circ$. Since idempotents in an inverse semigroup commute we have that
\begin{equation*}
\begin{split}
  \alpha\varepsilon_{1}^\circ\varepsilon_{2}^\circ\ldots\varepsilon_{k}^\circ= &
  \;\gamma^1\varepsilon_{1}^\circ\varepsilon_{2}^\circ\ldots\varepsilon_{k}^\circ=
  \;\gamma^1\varepsilon_{2}^\circ\varepsilon_{1}^\circ\ldots\varepsilon_{k}^\circ=
  \;\gamma^2\varepsilon_{2}^\circ\varepsilon_{1}^\circ\ldots\varepsilon_{k}^\circ=
  \;\gamma^2\varepsilon_{3}^\circ\varepsilon_{1}^\circ\ldots\varepsilon_{k}^\circ=
  \;\gamma^3\varepsilon_{3}^\circ\varepsilon_{1}^\circ\varepsilon_{2}^\circ \ldots\varepsilon_{k}^\circ= \dots=  \\
  =&\;\gamma^{k-1}\varepsilon_{k-1}^\circ\varepsilon_{2}^\circ\varepsilon_{1}^\circ \ldots\varepsilon_{k}^\circ=
  \gamma^{k-1}\varepsilon_{k}^\circ\varepsilon_{1}^\circ\varepsilon_{2}^\circ \ldots\varepsilon_{k-1}^\circ=
  \beta\varepsilon_{k}^\circ\varepsilon_{1}^\circ\varepsilon_{2}^\circ \ldots\varepsilon_{k-1}^\circ=
  \beta\varepsilon_{1}^\circ\varepsilon_{2}^\circ \ldots\varepsilon_{k}^\circ
\end{split}
\end{equation*}

$(\Leftarrow)$ Suppose that $\alpha\varepsilon_{i_1}^\circ\ldots\varepsilon_{i_k}^\circ= \beta\varepsilon_{i_1}^\circ\ldots\varepsilon_{i_k}^\circ$ for some idempotents $\varepsilon_{i_1}^\circ,\ldots,\varepsilon_{i_k}^\circ\in \mathscr{I\!O}\!_{\infty}(\mathbb{Z}^n_{\operatorname{lex}})$. Without loss of generality we can assume that $i_1=1,\ldots,i_k=k$. We put $\gamma^1=\alpha\varepsilon_{1}^\circ$, $\gamma^2=\alpha\varepsilon_{1}^\circ\varepsilon_{2}^\circ$, $\ldots$, $\gamma^k=\alpha\varepsilon_{1}^\circ\varepsilon_{2}^\circ\ldots\varepsilon_{k}^\circ=
\beta\varepsilon_{1}^\circ\varepsilon_{2}^\circ\ldots\varepsilon_{k}^\circ$, $\ldots$, $\gamma^{2k-1}=\beta\varepsilon_{1}^\circ\varepsilon_{2}^\circ$, $\gamma^{2k}=\beta\varepsilon_{1}^\circ$. Therefore we get that
\begin{equation*}
\alpha\sigma_{[1]}\gamma^1\sigma_{[2]}\gamma^2\sigma_{[2]}\ldots\sigma_{[k]}\gamma^k \sigma_{[k+1]}\gamma^{k+1}\sigma_{[k-2]}\ldots\sigma_{[2]}\gamma^{2k-1} \sigma_{[1]}\beta.
\end{equation*}
This implies that $\alpha(\sigma_{[1]}\circ\sigma_{[2]}\circ\ldots \sigma_{[k]}\circ\ldots\sigma_{[2]}\circ\sigma_{[1]})\beta$, end by Proposition~\ref{proposition-2.13} we have that $\alpha\sigma_{[i_1,\ldots,i_k]}\beta$.
\end{proof}

\begin{proposition}\label{proposition-2.16}
$\sigma_{[1,2,\ldots,n]}$ is the least group congruence on $\mathscr{I\!O}\!_{\infty}(\mathbb{Z}^n_{\operatorname{lex}})$, i.e., $\sigma_{[1,2,\ldots,n]}=\sigma$.
\end{proposition}

\begin{proof}
For every $i=1,\ldots,n$ the definition of $\sigma_{[i]}$ implies that $\sigma_{[i]}\subseteq \sigma$. Then by proposition~\ref{proposition-2.13} we have that $\sigma_{[1]}\circ\ldots\circ\sigma_{[n]}= \sigma_{[1]}\vee\ldots\vee\sigma_{[n]}$, and since the congruences form a lattice we conclude that $\sigma_{[1]}\circ\ldots\circ\sigma_{[n]}\subseteq\sigma$.

Let $\alpha=(\alpha_1,\alpha_2,\ldots,\alpha_n)$ and $\beta=(\beta_1,\beta_2,\ldots,\beta_n)$ be elements of the semigroup $\mathscr{I\!O}\!_{\infty}(\mathbb{Z}^n_{\operatorname{lex}})$ such that $\alpha\sigma\beta$. Then there exists an idempotent $\varepsilon=(\varepsilon_1,\varepsilon_2,\ldots,\varepsilon_n)$ such that $\alpha\varepsilon=\beta\varepsilon$, i.e.,
\begin{equation*}
(\alpha_1\varepsilon_1,\alpha_2\varepsilon_2,\ldots,\alpha_n\varepsilon_n)= (\beta_1\varepsilon_1,\beta_2\varepsilon_2,\ldots,\beta_n\varepsilon_n).
\end{equation*}
Now we put $\gamma^1=(\beta_1,\alpha_2,\ldots,\alpha_{n-1},\alpha_n)$, $\gamma^2=(\beta_1,\beta_2,\ldots,\alpha_{n-1},\alpha_n)$, $\ldots$, $\gamma^{n-1}=(\beta_1,\beta_2,\ldots,\beta_{n-1},\alpha_n)$. Then we have that $\alpha\sigma_{[1]}\gamma^1$, $\gamma^1\sigma_{[2]}\gamma^2$, $\ldots$, $\gamma^{n-1}\sigma_{[2]}\beta$. Therefore we get that $\sigma\subseteq \sigma_{[1]}\circ\ldots\circ\sigma_{[n]}$, and hence $\sigma= \sigma_{[1]}\circ\ldots\circ\sigma_{[n]}$.
\end{proof}

For every $i=1,\ldots,n$ we define a map $\pi^i\colon \mathscr{I\!O}\!_{\infty}(\mathbb{Z}^n_{\operatorname{lex}})\rightarrow \mathscr{I\!O}\!_{\infty}(\mathbb{Z}^n_{\operatorname{lex}})$ by the formula $(\alpha)\pi_i=\alpha_i^\circ$, i.e., $(\alpha_1,\ldots,\alpha_i,\ldots,\alpha_n)\pi^i= (\mathbb{I}_1,\ldots,\mathbb{I}_{i-1},\alpha_i,\mathbb{I}_{i+1},\ldots,\mathbb{I}_n)$. Simple verifications show that such defined map $\pi^i\colon \mathscr{I\!O}\!_{\infty}(\mathbb{Z}^n_{\operatorname{lex}})\rightarrow \mathscr{I\!O}\!_{\infty}(\mathbb{Z}^n_{\operatorname{lex}})$ is a homomorphism. Let ${\pi^i}^\sharp$ be the congruence on $\mathscr{I\!O}\!_{\infty}(\mathbb{Z}^n_{\operatorname{lex}})$ which is generated by the homomorphism $\pi^i$.

Let $S$ be an inverse semigroup. For any congruence $\rho$ on $S$ we define a congruence $\rho_{\min}$ on $S$ as follows:
\begin{equation*}
    a\rho_{\min}b \qquad \hbox{if and only if} \qquad ae=be \quad \hbox{for some} \quad e\in E(S), \quad e\rho a^{-1}a\rho b^{-1}b,
\end{equation*}
(see: \cite[Section~III.2]{Petrich1984}).

\begin{proposition}\label{proposition-2.17}
${\pi^i}^\sharp_{\min}= \sigma_{[1]}\circ\ldots\circ\sigma_{[i-1]}\circ\sigma_{[i+1]}\circ\ldots \circ\sigma_{[n]}$ for every $i=1,\ldots,n$.
\end{proposition}

\begin{proof}
$(\Leftarrow)$ Suppose that $\alpha(\sigma_{[1]}\circ\ldots \circ\sigma_{[i-1]}\circ\sigma_{[i+1]}\circ \ldots\circ\sigma_{[n]})\beta$ in $\mathscr{I\!O}\!_{\infty}(\mathbb{Z}^n_{\operatorname{lex}})$ for some $\alpha=(\alpha_1,\ldots,\alpha_n)$ and $\beta=(\beta_1,\ldots,\beta_n)$. Then by Proposition~\ref{proposition-2.15} we have that $\alpha\varepsilon_{1}^\circ\ldots\varepsilon_{i-1}^\circ \varepsilon_{i+1}^\circ\ldots\varepsilon_{n}^\circ= \beta\varepsilon_{1}^\circ\ldots\varepsilon_{i-1}^\circ \varepsilon_{i+1}^\circ\ldots\varepsilon_{n}^\circ$ for some idempotent $\varepsilon=(\varepsilon_{1},\ldots,\varepsilon_{i-1},\mathbb{I}_i, \varepsilon_{i+1},\ldots,\varepsilon_{n})$, i.e., $\alpha\varepsilon=\beta\varepsilon$. Then we have that $\alpha_i=\beta_i$, and hence $\alpha\varepsilon^*=\beta\varepsilon^*$ for $\varepsilon^*=(\varepsilon_{1},\ldots,\varepsilon_{i-1},\alpha_i^{-1}\alpha_i, \varepsilon_{i+1},\ldots,\varepsilon_{n})$. It is obvious that $\varepsilon^*{\pi^i}^\sharp\alpha^{-1}\alpha{\pi^i}^\sharp\beta^{-1}\beta$. This implies that $\sigma_{[1]}\circ\ldots \circ\sigma_{[i-1]} \circ\sigma_{[i+1]} \circ\ldots \circ\sigma_{[n]}\subseteq{\pi^i}^\sharp_{\min}$.

$(\Rightarrow)$ Suppose that $\alpha{\pi^i}^\sharp_{\min}\beta$ in $\mathscr{I\!O}\!_{\infty}(\mathbb{Z}^n_{\operatorname{lex}})$ for some $\alpha=(\alpha_1,\ldots,\alpha_n)$ and $\beta=(\beta_1,\ldots,\beta_n)$. The there exists an idempotent $\varepsilon=(\varepsilon_{1},\ldots,\varepsilon_{n})$ in $\mathscr{I\!O}\!_{\infty}(\mathbb{Z}^n_{\operatorname{lex}})$ such that $\alpha\varepsilon=\beta\varepsilon$ and $\varepsilon{\pi^i}^\sharp \alpha^{-1}\alpha{\pi^i}^\sharp\beta^{-1}\beta$. The last two equalities imply that $\alpha_i^{-1}\alpha_i=\beta_i^{-1}\beta_i=\varepsilon_i$. This and the equality $\alpha\varepsilon=\beta\varepsilon$ imply that $\alpha_i\varepsilon_i= \beta_i\varepsilon_i$ and hence $\alpha_i=\alpha_i\alpha_i^{-1}\alpha_i= \alpha_i\varepsilon_i= \beta_i\varepsilon_i= \beta_i\beta_i^{-1}\beta_i=\beta_i$. Therefore we have that $\alpha\varepsilon^*=\beta\varepsilon^*$, where $\varepsilon^*=(\varepsilon_1,\ldots,\varepsilon_{i-1},\mathbb{I}_i, \varepsilon_{i+1},\ldots,\varepsilon_n)$, i.e., $\alpha\varepsilon_{1}^\circ\ldots\varepsilon_{i-1}^\circ \varepsilon_{i+1}^\circ\ldots\varepsilon_{n}^\circ= \beta\varepsilon_{1}^\circ\ldots\varepsilon_{i-1}^\circ \varepsilon_{i+1}^\circ\ldots\varepsilon_{n}^\circ$. Then by Proposition~\ref{proposition-2.15} we have that $\alpha(\sigma_{[1]}\circ\ldots \circ\sigma_{[i-1]}\circ\sigma_{[i+1]}\circ \ldots\circ\sigma_{[n]})\beta$ in $\mathscr{I\!O}\!_{\infty}(\mathbb{Z}^n_{\operatorname{lex}})$. This implies that ${\pi^i}^\sharp_{\min}\subseteq \sigma_{[1]}\circ\ldots\circ\sigma_{[i-1]}\circ\sigma_{[i+1]}\circ\ldots \circ\sigma_{[n]}$.
\end{proof}

For every $\alpha\in\mathscr{I\!O}\!_{\infty}(\mathbb{Z}^n_{\operatorname{lex}})$ and any $(i,j)\in\operatorname{dom}\alpha\subseteq L_n\times\mathbb{Z}$ according to Lemma~\ref{lemma-2.1} we denote $(i,j)\alpha=(i,(j)\alpha_i)$.

\begin{proposition}\label{proposition-2.18}
Let $\{i_1,\ldots,i_k\}\subseteq\{1,\ldots,n\}$ be any collection of distinct indices, $k\le n$. Then $\alpha\sigma_{[i_1,\ldots,i_k]}\beta$ in $\mathscr{I\!O}\!_{\infty}(\mathbb{Z}^n_{\operatorname{lex}})$ if and only if the following conditions hold:
\begin{itemize}
  \item[$(i)$] there exists a positive integer $p$ such that $(j)\alpha_i=(j)\beta_i$ for all integers $j$ with $|j|\ge p$ and all $i=i_1,\ldots,i_k$;
  \item[$(ii)$] $\operatorname{dom}\alpha\cap\big(\big(\{1,\ldots n\}\setminus \{i_1,\ldots,i_k\}\big)\times \mathbb{Z}\big)= \operatorname{dom}\beta\cap\big(\big(\{1,\ldots n\}\setminus \{i_1,\ldots,i_k\}\big)\times \mathbb{Z}\big)$ and $(j)\alpha_i=(j)\beta_i$ for all $(i,j)\in\operatorname{dom}\alpha\cap\big(\big(\{1,\ldots n\}\setminus \{i_1,\ldots,i_k\}\big)\times \mathbb{Z}\big)$.
\end{itemize}
\end{proposition}

\begin{proof}
Without loss of generality we can assume that $i_1=1,\ldots,i_k=k$.

$(\Rightarrow)$ Suppose that $\alpha\sigma_{[1,\ldots,k]}\beta$ in $\mathscr{I\!O}\!_{\infty}(\mathbb{Z}^n_{\operatorname{lex}})$. Then there exist idempotents $\varepsilon_1^\circ,\ldots,\varepsilon_k^\circ$ in $\mathscr{I\!O}\!_{\infty}(\mathbb{Z}^n_{\operatorname{lex}})$ such that $\alpha\varepsilon_1^\circ\ldots\varepsilon_k^\circ= \beta\varepsilon_1^\circ\ldots\varepsilon_k^\circ$. This implies assertion $(ii)$.

We observe that the definition of the idempotent $\varepsilon_i^\circ$, $i=1,\ldots,n$, implies that the restriction $\varepsilon_i^\circ|_{(L_n\setminus\{i\})\times \mathbb{Z}}$ is an identity map of the set $(L_n\setminus\{i\})\times \mathbb{Z}$. Therefore there exists a positive integer $p_i$ such that
\begin{equation*}
    \big(\{(i,j)\colon |j|\ge p_i\}\big)\cup(L_n\setminus\{i\})\times \mathbb{Z}\subseteq \operatorname{dom}\varepsilon_i^\circ.
\end{equation*}
We put $p=\max\{p_1,\ldots,p_k\}$ and $p$ requested as in $(ii)$.

$(\Leftarrow)$ Suppose that assertions $(i)$ and $(ii)$ hold. By $\operatorname{Id}_M$ we denote the partial identity map of the subset $M$ for any $M\subseteq L_n\times\mathbb{Z}$. For every $i=1,\ldots,n$ we put
\begin{equation*}
\varepsilon_i^\circ=\operatorname{Id}_{\{(i,j)\mid |j|\ge p\}}\cup \operatorname{Id}_{(L_n\setminus\{i\})\times \mathbb{Z}}.
\end{equation*}
Simple verifications show that $\alpha\varepsilon_1^\circ\ldots\varepsilon_k^\circ= \beta\varepsilon_1^\circ\ldots\varepsilon_k^\circ$ and hence $\alpha\sigma_{[1,\ldots,k]}\beta$ in $\mathscr{I\!O}\!_{\infty}(\mathbb{Z}^n_{\operatorname{lex}})$.
\end{proof}

\section{Generators and automorphisms  of the semigroup
$\mathscr{I\!O}\!_{\infty}(\mathbb{Z}^n_{\operatorname{lex}})$}\label{section-3}

We put $\mathscr{O}\!_{\infty}^{\;0}(\mathbb{Z})= \{\alpha\in\mathscr{I\!O}\!_{\infty}(\mathbb{Z})\colon \operatorname{dom}\alpha=\mathbb{Z}\}$ and $\mathscr{O}\!_{\infty}^{\;[0]}(\mathbb{Z})= \{\alpha\in\mathscr{I\!O}\!_{\infty}(\mathbb{Z})\colon \operatorname{ran}\alpha=\mathbb{Z}\}$.

\begin{proposition}\label{proposition-3.1}
$\mathscr{O}\!_{\infty}^{\;0}(\mathbb{Z})$ and $\mathscr{O}\!_{\infty}^{\;[0]}(\mathbb{Z})$ are antiisomorphic subsemigroups of $\mathscr{I\!O}\!_{\infty}(\mathbb{Z})$.
\end{proposition}

\begin{proof}
Simple verifications imply that $\mathscr{O}\!_{\infty}^{\;0}(\mathbb{Z})$ and $\mathscr{O}\!_{\infty}^{\;[0]}(\mathbb{Z})$ are subsemigroups of the semigroup $\mathscr{I\!O}\!_{\infty}(\mathbb{Z})$. We define $\mathfrak{i}\colon \mathscr{O}\!_{\infty}^{\;0}(\mathbb{Z})\rightarrow \mathscr{O}\!_{\infty}^{\;[0]}(\mathbb{Z})$ by the formula $(\alpha)\mathfrak{i}=\alpha^{-1}$. It is obvious that so defined map $\mathfrak{i}\colon \mathscr{O}\!_{\infty}^{\;0}(\mathbb{Z})\rightarrow \mathscr{O}\!_{\infty}^{\;[0]}(\mathbb{Z})$ is surjective and since the map $\mathfrak{i}$ is the restriction of inversion of the inverse semigroup $\mathscr{I\!O}\!_{\infty}(\mathbb{Z})$ onto the subsemigroup $\mathscr{O}\!_{\infty}^{\;0}(\mathbb{Z})$ we get that it is an antiisomorphism.
\end{proof}

It is obvious that the group of units of the semigroup $\mathscr{I\!O}\!_{\infty}(\mathbb{Z})$ is isomorphic to the group of units of $\mathscr{O}\!_{\infty}^{\;0}(\mathbb{Z})$ ($\mathscr{O}\!_{\infty}^{\;[0]}(\mathbb{Z})$), and moreover by Proposition~\ref{proposition-2.6} it is isomorphic to the additive group of integers $\mathbb{Z}(+)$.

Simple observations imply the following proposition.

\begin{proposition}\label{proposition-3.2}
The subsemigroups $\mathscr{O}\!_{\infty}^{\;0}(\mathbb{Z})$ and $\mathscr{O}\!_{\infty}^{\;[0]}(\mathbb{Z})$ (as a subset) generate the inverse semigroup $\mathscr{I\!O}\!_{\infty}(\mathbb{Z})$, and moreover $\mathscr{O}\!_{\infty}^{\;[0]}(\mathbb{Z})\cdot \mathscr{O}\!_{\infty}^{\;0}(\mathbb{Z})= \mathscr{I\!O}\!_{\infty}(\mathbb{Z})$.
\end{proposition}

For an arbitrary integer $k$ we define the maps $\varepsilon_k\colon \mathbb{Z}\rightarrow\mathbb{Z}$ and $\varsigma_k\colon\mathbb{Z}\rightarrow\mathbb{Z}$ by the formulae
\begin{equation*}
    (i)\varepsilon_k=
    \left\{
      \begin{array}{cl}
        i+1, & \hbox{if~} i>k;\\
        i, & \hbox{if~} i\leqslant k
      \end{array}
    \right.
\qquad \hbox{and} \qquad (i)\varsigma_k=i+k.
\end{equation*}
Obviously that $\varepsilon_k\in \mathscr{O}\!_{\infty}^{\;0}(\mathbb{Z})$, $\varepsilon_k^{-1}\in\mathscr{O}\!_{\infty}^{\;[0]}(\mathbb{Z})$ and $\varsigma_k$ is an element of the group of units of $\mathscr{I\!O}\!_{\infty}(\mathbb{Z})$, for every $k\in\mathbb{Z}$.

\begin{proposition}\label{proposition-3.3}
The set $\{\varepsilon_0,\varsigma_1\}$ generates the semigroup $\mathscr{I\!O}\!_{\infty}(\mathbb{Z})$ as an inverse semigroup.
\end{proposition}

\begin{proof}
By Theorem~4.2 of \cite{Doroshenko2005} the set $\{\varepsilon_0,\varsigma_1\}$ generates the semigroup $\mathscr{O}\!_{\infty}^{\;0}(\mathbb{Z})$ and hence by Proposition~\ref{proposition-3.1} we get that the set $\{\varepsilon_0^{-1},\varsigma_1^{-1}=\varsigma_{-1}\}$ generates the semigroup $\mathscr{O}\!_{\infty}^{\;[0]}(\mathbb{Z})$. Next, Proposition~\ref{proposition-3.2} implies the statement of the proposition.
\end{proof}

Proposition~\ref{proposition-3.3} implies the following

\begin{theorem}\label{theorem-3.4}
For every integer $k$ the set $\{\varepsilon_k,\varsigma_1\}$ generates the semigroup $\mathscr{I\!O}\!_{\infty}(\mathbb{Z})$ as an inverse semigroup and hence $\mathscr{I\!O}\!_{\infty}(\mathbb{Z})$ is finitely generated. Moreover, every minimal system of generators of the semigroup $\mathscr{I\!O}\!_{\infty}(\mathbb{Z})$ (as an inverse semigroup) has the form $\{\varepsilon_k,\varsigma_{i_1},\ldots,\varsigma_{i_m}\}$, where $k$ is an arbitrary integer and the set of indices $i_1,\ldots, i_m$ is a
minimal system of generators of the semigroup $\mathbb{Z}(+)$ (as a group).
\end{theorem}

\begin{remark}\label{remark-3.5}
It is obvious that the $\{1\}$ and $\{-1\}$ are the minimal systems of generators of the additive group group of integers $\mathbb{Z}(+)$ as a group.
\end{remark}

For an arbitrary positive integer $n$ we put $\mathscr{O}\!_{\infty}^{\;0}(\mathbb{Z}^n_{\operatorname{lex}})= \big\{\alpha\in\mathscr{I\!O}\!_{\infty}(\mathbb{Z}^n_{\operatorname{lex}})\colon \operatorname{dom}\alpha=L_n\times\mathbb{Z}\big\}$ and $\mathscr{O}\!_{\infty}^{\;[0]}(\mathbb{Z}^n_{\operatorname{lex}})= \big\{\alpha\in\mathscr{I\!O}\!_{\infty}(\mathbb{Z}^n_{\operatorname{lex}})\colon \operatorname{ran}\alpha=L_n\times\mathbb{Z}\big\}$.

The proof of the following proposition is similar to the proof of Proposition~\ref{proposition-3.1}.

\begin{proposition}\label{proposition-3.6}
$\mathscr{O}\!_{\infty}^{\;0}(\mathbb{Z}^n_{\operatorname{lex}})$ and $\mathscr{O}\!_{\infty}^{\;[0]}(\mathbb{Z}^n_{\operatorname{lex}})$ are antiisomorphic subsemigroups of $\mathscr{I\!O}\!_{\infty}(\mathbb{Z}^n_{\operatorname{lex}})$.
\end{proposition}

Proposition~\ref{proposition-2.3}$(iv)$ implies the following:

\begin{proposition}\label{proposition-3.7}
For every positive integer $n$ the semigroup  $\mathscr{O}\!_{\infty}^{\;0}(\mathbb{Z}^n_{\operatorname{lex}})$ $\big(\hbox{resp., } \mathscr{O}\!_{\infty}^{\;[0]}(\mathbb{Z}^n_{\operatorname{lex}})\big)$ is isomorphic to the direct power $\left(\mathscr{O}\!_{\infty}^{\;0}(\mathbb{Z})\right)^n$ $\big(\hbox{resp., }\big(\mathscr{O}\!_{\infty}^{\;[0]}(\mathbb{Z})\big)^n\big)$.
\end{proposition}

Also we observe that by Proposition~\ref{proposition-2.6} the groups of units of the semigroups $\mathscr{O}\!_{\infty}^{\;0}(\mathbb{Z}^n_{\operatorname{lex}})$ and $\mathscr{O}\!_{\infty}^{\;[0]}(\mathbb{Z}^n_{\operatorname{lex}})$ are isomorphic to the direct power $\big(\mathbb{Z}(+)\big)^n$.

Propositions~\ref{proposition-2.3}$(iv)$ and~\ref{proposition-3.2} imply the following:

\begin{proposition}\label{proposition-3.8}
The subsemigroups $\mathscr{O}\!_{\infty}^{\;0}(\mathbb{Z}^n_{\operatorname{lex}})$ and $\mathscr{O}\!_{\infty}^{\;[0]}(\mathbb{Z}^n_{\operatorname{lex}})$ (as a subset) generate the inverse semigroup $\mathscr{I\!O}\!_{\infty}(\mathbb{Z}^n_{\operatorname{lex}})$, and moreover $\mathscr{O}\!_{\infty}^{\;[0]}(\mathbb{Z}^n_{\operatorname{lex}})\cdot \mathscr{O}\!_{\infty}^{\;0}(\mathbb{Z}^n_{\operatorname{lex}})= \mathscr{I\!O}\!_{\infty}(\mathbb{Z}^n_{\operatorname{lex}})$.
\end{proposition}

For an arbitrary positive integer $n$ and any integers $k$ and $j$ such that $j=1,\ldots,n$, we define the maps $\varepsilon_{k[j]}\colon L_n\times\mathbb{Z}\rightarrow L_n\times\mathbb{Z}$ and $\varsigma_{k[j]}\colon L_n\times\mathbb{Z}\rightarrow L_n\times\mathbb{Z}$ by the formulae
\begin{equation*}
    (m,i)\varepsilon_{k[j]}=
    \left\{
      \begin{array}{cl}
        (m,i+1), & \hbox{if~} m=j \hbox{~and~} i>k;\\
        (m,i), & \hbox{if~} m=j \hbox{~and~} i\leqslant k;\\
        (m,i), & \hbox{if~} m\neq j
      \end{array}
    \right.
\qquad \hbox{and} \qquad (m,i)\varsigma_{k[j]}=
    \left\{
      \begin{array}{cl}
        (m,i+k), & \hbox{if~} m=j;\\
        (m,i), & \hbox{if~} m\neq j.
      \end{array}
    \right.
\end{equation*}
Obviously that $\varepsilon_{k[j]}\in \mathscr{O}\!_{\infty}^{\;0}(\mathbb{Z}^n_{\operatorname{lex}})$, $\varepsilon_{k[j]}^{-1}\in \mathscr{O}\!_{\infty}^{\;[0]}(\mathbb{Z}^n_{\operatorname{lex}})$ and $\varsigma_{k[j]}$ is an element of the group of units of $\mathscr{I\!O}\!_{\infty}(\mathbb{Z}^n_{\operatorname{lex}})$, for any $k\in\mathbb{Z}$ and $j=1,\ldots,n$.

Theorem~\ref{theorem-3.4}, Propositions~\ref{proposition-3.7} and~\ref{proposition-3.8} imply the following theorem.

\begin{theorem}\label{theorem-3.9}
For every positive integer $n$ and any $n$-ordered collection of integers $(k_1,\ldots,k_n)$ the set  $\{\varepsilon_{k_1[1]},\ldots,\varepsilon_{k_n[n]}, \varsigma_{1[1]},\ldots,\varsigma_{1[n]}\}$ generates the semigroup $\mathscr{I\!O}\!_{\infty}(\mathbb{Z}^n_{\operatorname{lex}})$ as an inverse semigroup and hence $\mathscr{I\!O}\!_{\infty}(\mathbb{Z}^n_{\operatorname{lex}})$ is finitely generated.
\end{theorem}

\begin{remark}\label{remark-3.10}
We observe that for every positive integer $n$ and any $2n$-ordered collection of integers $(k_1,\ldots,k_n,k_{n+1},\ldots,k_{2n})$ the set
\begin{equation*}
\big\{\varepsilon_{k_1[1]},\ldots,\varepsilon_{k_n[n]}, \varepsilon_{k_{n+1}[1]}^{-1},\ldots,\varepsilon_{k_{2n}[n]}^{-1}, \varsigma_{1[1]},\ldots,\varsigma_{1[n]},\varsigma_{-1[1]},\ldots, \varsigma_{-1[n]}\big\}
\end{equation*}
generates the semigroup $\mathscr{I\!O}\!_{\infty}(\mathbb{Z}^n_{\operatorname{lex}})$ as a semigroup in the general case.
\end{remark}

\begin{proposition}\label{proposition-3.11}
Let $\mathfrak{f}\colon \mathscr{I\!O}\!_{\infty}(\mathbb{Z})\rightarrow \mathscr{I\!O}\!_{\infty}(\mathbb{Z})$ be any automorphism of the semigroup $\mathscr{I\!O}\!_{\infty}(\mathbb{Z})$. Then $\big(\mathscr{O}\!_{\infty}^{\;0}(\mathbb{Z})\big)\mathfrak{f}= \mathscr{O}\!_{\infty}^{\;0}(\mathbb{Z})$ and $\big(\mathscr{O}\!_{\infty}^{\;[0]}(\mathbb{Z})\big)\mathfrak{f}= \mathscr{O}\!_{\infty}^{\;[0]}(\mathbb{Z})$, and moreover the restrictions $\mathfrak{f}|_{\mathscr{O}\!_{\infty}^{\;0}(\mathbb{Z})}\colon \mathscr{O}\!_{\infty}^{\;0}(\mathbb{Z})\rightarrow \mathscr{O}\!_{\infty}^{\;0}(\mathbb{Z})$ and $\mathfrak{f}|_{\mathscr{O}\!_{\infty}^{\;[0]}(\mathbb{Z})}\colon \mathscr{O}\!_{\infty}^{\;[0]}(\mathbb{Z})\rightarrow \mathscr{O}\!_{\infty}^{\;[0]}(\mathbb{Z})$ are automorphisms of the semigroups $\mathscr{O}\!_{\infty}^{\;0}(\mathbb{Z})$ and $\mathscr{O}\!_{\infty}^{\;[0]}(\mathbb{Z})$, respectively.
\end{proposition}

\begin{proof}
Fix an arbitrary element $\alpha\in\mathscr{O}\!_{\infty}^{\;0}(\mathbb{Z})$. Then we have that $\alpha^{-1}\in\mathscr{O}\!_{\infty}^{\;[0]}(\mathbb{Z})$ and hence $\alpha\alpha^{-1}=\mathbb{I}$ is the unity element of the semigroup $\mathscr{I\!O}\!_{\infty}(\mathbb{Z})$. Suppose to the contrary that there exists $\alpha\in\mathscr{O}\!_{\infty}^{\;0}(\mathbb{Z})$ such that $(\alpha)\mathfrak{f}\notin\mathscr{O}\!_{\infty}^{\;0}(\mathbb{Z})$. Then we have that
\begin{equation*}
    \mathbb{I}=(\mathbb{I})\mathfrak{f}=(\alpha\alpha^{-1})\mathfrak{f}= (\alpha)\mathfrak{f}(\alpha^{-1})\mathfrak{f}.
\end{equation*}
The last formula implies that $\operatorname{dom}\mathbb{I}\neq\mathbb{Z}$ because $(\alpha)\mathfrak{f}\notin\mathscr{O}\!_{\infty}^{\;0}(\mathbb{Z})$, a contradiction. The obtained contradiction implies that $(\alpha)\mathfrak{f}\in\mathscr{O}\!_{\infty}^{\;0}(\mathbb{Z})$ and hence $\big(\mathscr{O}\!_{\infty}^{\;0}(\mathbb{Z})\big)\mathfrak{f}\subseteq  \mathscr{O}\!_{\infty}^{\;0}(\mathbb{Z})$. Suppose there exists $\beta\in \mathscr{I\!O}\!_{\infty}(\mathbb{Z})\setminus \mathscr{O}\!_{\infty}^{\;0}(\mathbb{Z})$ such that $(\beta)\mathfrak{f}\in\mathscr{O}\!_{\infty}^{\;0}(\mathbb{Z})$. Then the inverse map $\mathfrak{f}^{-1}$ of $\mathfrak{f}$ is an automorphism of the semigroup $\mathscr{I\!O}\!_{\infty}(\mathbb{Z})$, and by previous arguments we get that $\mathscr{O}\!_{\infty}^{\;0}(\mathbb{Z})\not\ni\beta= (\beta)\mathfrak{f}\mathfrak{f}^{-1}\in \mathscr{O}\!_{\infty}^{\;0}(\mathbb{Z})$, a contradiction. Thus, the equality $\big(\mathscr{O}\!_{\infty}^{\;0}(\mathbb{Z})\big)\mathfrak{f}= \mathscr{O}\!_{\infty}^{\;0}(\mathbb{Z})$ holds.

The proof of the equality $\big(\mathscr{O}\!_{\infty}^{\;[0]}(\mathbb{Z})\big)\mathfrak{f}= \mathscr{O}\!_{\infty}^{\;[0]}(\mathbb{Z})$ is similar. The last assertion follows from the first part of the proof.
\end{proof}

Since by Proposition~\ref{proposition-3.3} the elements $\varepsilon_0,\varsigma_1\in \mathscr{O}\!_{\infty}^{\;0}(\mathbb{Z})$ (resp., $\varepsilon_0^{-1},\varsigma_1\in \mathscr{O}\!_{\infty}^{\;[0]}(\mathbb{Z})$) generate $\mathscr{I\!O}\!_{\infty}(\mathbb{Z})$ as an inverse semigroup, the following proposition holds.

\begin{proposition}\label{proposition-3.12}
For an arbitrary automorphism $\mathfrak{f}\colon \mathscr{O}\!_{\infty}^{\;0}(\mathbb{Z})\rightarrow \mathscr{O}\!_{\infty}^{\;0}(\mathbb{Z})$ (resp., $\mathfrak{f}\colon \mathscr{O}\!_{\infty}^{\;[0]}(\mathbb{Z})\rightarrow \mathscr{O}\!_{\infty}^{\;[0]}(\mathbb{Z})$) of the semigroup $\mathscr{O}\!_{\infty}^{\;0}(\mathbb{Z})$ (resp., $\mathscr{O}\!_{\infty}^{\;[0]}(\mathbb{Z})$) there exists a unique automorphism
$\overline{\mathfrak{f}}\colon \mathscr{I\!O}\!_{\infty}(\mathbb{Z})\rightarrow \mathscr{I\!O}\!_{\infty}(\mathbb{Z})$ of the semigroup $\mathscr{I\!O}\!_{\infty}(\mathbb{Z})$ such that
$\overline{\mathfrak{f}}|_{\mathscr{O}\!_{\infty}^{\;0}(\mathbb{Z})}=\mathfrak{f}$ (resp., $\overline{\mathfrak{f}}|_{\mathscr{O}\!_{\infty}^{\;[0]}(\mathbb{Z})}= \mathfrak{f}$).
\end{proposition}

By Theorem~9 from \cite{Doroshenko2009} every automorphism of the semigroup $\mathscr{O}\!_{\infty}^{\;0}(\mathbb{Z})$ is inner and moreover the group of automorphisms of $\mathscr{O}\!_{\infty}^{\;0}(\mathbb{Z})$ is isomorphic to the additive group of integers $\mathbb{Z}(+)$. Then Proposition~\ref{proposition-3.1} implies the following corollary.

\begin{corollary}\label{corollary-3.13}
Every automorphism of the semigroup $\mathscr{O}\!_{\infty}^{\;[0]}(\mathbb{Z})$ is inner and moreover the group of automorphisms of $\mathscr{O}\!_{\infty}^{\;[0]}(\mathbb{Z})$ is isomorphic to the additive group of integers $\mathbb{Z}(+)$.
\end{corollary}

\begin{theorem}\label{theorem-3.14}
Every automorphism of the semigroup $\mathscr{I\!O}\!_{\infty}(\mathbb{Z})$ is inner and moreover the group of automorphisms of $\mathscr{I\!O}\!_{\infty}(\mathbb{Z})$ is isomorphic to the additive group of integers $\mathbb{Z}(+)$.
\end{theorem}

\begin{proof}
Let
$\mathfrak{f}\colon \mathscr{I\!O}\!_{\infty}(\mathbb{Z})\rightarrow \mathscr{I\!O}\!_{\infty}(\mathbb{Z})$ be an arbitrary automorphism of the semigroup $\mathscr{I\!O}\!_{\infty}(\mathbb{Z})$. Then Theorem~9 from \cite{Doroshenko2009} and Corollary~\ref{corollary-3.13} imply there exist integers $i$ and $j$ such that $(\alpha)\mathfrak{f}=\varsigma_i\alpha\varsigma_i^{-1}$ and $(\beta)\mathfrak{f}=\varsigma_j\beta\varsigma_j^{-1}$ for all $\alpha\in \mathscr{O}\!_{\infty}^{\;0}(\mathbb{Z})$ and $\beta\in \mathscr{O}\!_{\infty}^{\;[0]}(\mathbb{Z})$. Next we shall show that $i=j$. Suppose the contrary that $i\neq j$. We fix an arbitrary $\alpha\in \mathscr{O}\!_{\infty}^{\;0}(\mathbb{Z})$ such that $\alpha$ is not an element of the group of units of $\mathscr{O}\!_{\infty}^{\;0}(\mathbb{Z})$. Then $\alpha^{-1}\in \mathscr{O}\!_{\infty}^{\;[0]}(\mathbb{Z})$ and $\alpha\alpha^{-1}=\mathbb{I}$ is unit of the semigroup $\mathscr{I\!O}\!_{\infty}(\mathbb{Z})$. Now, we have that
\begin{equation*}
    \mathbb{I}=(\mathbb{I})\mathfrak{f}=(\alpha\alpha^{-1})\mathfrak{f}= (\alpha)\mathfrak{f}(\alpha^{-1})\mathfrak{f}= \varsigma_i\alpha\varsigma_i^{-1} \varsigma_j\alpha^{-1}\varsigma_j^{-1}.
\end{equation*}
Since $i\neq j$ we get that $\operatorname{dom}(\alpha\varsigma_i^{-1} \varsigma_j\alpha^{-1})\neq \mathbb{Z}$ and hence $\operatorname{dom}\mathbb{I}=\operatorname{dom}(\varsigma_i\alpha\varsigma_i^{-1} \varsigma_j\alpha^{-1}\varsigma_j^{-1})\neq \mathbb{Z}$, a contradiction. The obtained contradiction implies that $i=j$. Now, Theorem~\ref{theorem-3.4}, Theorem~9 of\cite{Doroshenko2009} and Corollary~\ref{corollary-3.13} complete the proof of the theorem.
\end{proof}

The following example implies that for every integer $n\geqslant 2$ the semigroup $\mathscr{I\!O}\!_{\infty}(\mathbb{Z}^n_{\operatorname{lex}})$ has a non-inner automorphism.

\begin{example}\label{example-3.15}
We define the map $\mathfrak{h}\colon \mathscr{I\!O}\!_{\infty}(\mathbb{Z}^2_{\operatorname{lex}})\rightarrow \mathscr{I\!O}\!_{\infty}(\mathbb{Z}^2_{\operatorname{lex}})$ in the following way. By Proposition~\ref{proposition-2.3}$(iv)$ we identify the semigroup $\mathscr{I\!O}\!_{\infty}(\mathbb{Z}^2_{\operatorname{lex}})$ with $\left(\mathscr{I\!O}\!_{\infty}(\mathbb{Z})\right)^{2}$ and put $(\alpha_1,\alpha_2)\mathfrak{h}=(\alpha_2,\alpha_1)$. It is obvious that so defined map $\mathfrak{h}$ is an automorphism of the semigroup $\mathscr{I\!O}\!_{\infty}(\mathbb{Z}^2_{\operatorname{lex}})$. It is easy to see that the restriction of an inner automorphism of an arbitrary monoid onto its group of units is an inner automorphism. Therefore it is complete to show that the restriction $\mathfrak{h}|_{H(\mathbb{I})}\colon H(\mathbb{I})\rightarrow H(\mathbb{I})$ is not an inner automorphism. Suppose to the contrary: the automorphism $\mathfrak{h}\colon \mathscr{I\!O}\!_{\infty}(\mathbb{Z}^2_{\operatorname{lex}})\rightarrow \mathscr{I\!O}\!_{\infty}(\mathbb{Z}^2_{\operatorname{lex}})$ is inner. By Proposition~\ref{proposition-2.6} the group of units of $\left(\mathscr{I\!O}\!_{\infty}(\mathbb{Z})\right)^{2}$ is isomorphic to $\left(\mathbb{Z}(+)\right)^2$, and since the group $\left(\mathscr{I\!O}\!_{\infty}(\mathbb{Z})\right)^{2}$ is commutative we get that the restriction $\mathfrak{h}|_{H(\mathbb{I})}\colon H(\mathbb{I})\rightarrow H(\mathbb{I})$ is trivial, a contradiction. The obtained contradiction implies that the automorphism $\mathfrak{h}\colon \mathscr{I\!O}\!_{\infty}(\mathbb{Z}^2_{\operatorname{lex}})\rightarrow \mathscr{I\!O}\!_{\infty}(\mathbb{Z}^2_{\operatorname{lex}})$ is not inner.

Also, the above implies that in the case when $n>2$ we have that the automorphism $\mathfrak{h}\colon \mathscr{I\!O}\!_{\infty}(\mathbb{Z}^n_{\operatorname{lex}})\rightarrow \mathscr{I\!O}\!_{\infty}(\mathbb{Z}^n_{\operatorname{lex}})$ defined by the formula $(\alpha_1,\alpha_2,\alpha_3,\ldots,\alpha_n)\mathfrak{h}= (\alpha_2,\alpha_1,\alpha_3,\ldots,\alpha_n)$ is not inner.
\end{example}

\section{On topologizations of the semigroup
$\mathscr{I}^{\!\nearrow}_{\infty}(\mathbb{Z}^n_{\operatorname{lex}})$}

\begin{theorem}\label{theorem-4.1}
Every Baire topology $\tau$ on
$\mathscr{I\!O}\!_{\infty}(\mathbb{Z}^n_{\operatorname{lex}})$ such that
$(\mathscr{I\!O}\!_{\infty}(\mathbb{Z}^n_{\operatorname{lex}}),\tau)$ is a
Hausdorff semitopological semigroup is discrete.
\end{theorem}

\begin{proof}
If no point in $\mathscr{I\!O}\!_{\infty}(\mathbb{Z}^n_{\operatorname{lex}})$ is
isolated,  then since
$(\mathscr{I\!O}\!_{\infty}(\mathbb{Z}^n_{\operatorname{lex}}),\tau)$ is Hausdorff,
it follows that $\{\alpha\}$ is nowhere dense for all
$\alpha\in\mathscr{I\!O}\!_{\infty}(\mathbb{Z}^n_{\operatorname{lex}}))$. But, if
this is the case, then since
$\mathscr{I\!O}\!_{\infty}(\mathbb{Z}^n_{\operatorname{lex}})$ is countable it
cannot be a Baire space. Hence
$\mathscr{I\!O}\!_{\infty}(\mathbb{Z}^n_{\operatorname{lex}})$ contains an isolated
point $\mu$. If
$\gamma\in\mathscr{I\!O}\!_{\infty}(\mathbb{Z}^n_{\operatorname{lex}})$ is
arbitrary, then by Proposition~\ref{proposition-2.4}~$(ix)$, there
exist $\alpha,\beta\in
\mathscr{I\!O}\!_{\infty}(\mathbb{Z}^n_{\operatorname{lex}})$ such that
$\alpha\cdot\gamma\cdot\beta=\mu$. The map
$f\colon\chi\mapsto\alpha\cdot\chi\cdot\beta$ is continuous and so
$(\{\mu\})f^{-1}$ is open. By Proposition~\ref{proposition-2.8},
$(\{\mu\})f^{-1}$ is finite and since
$(\mathscr{I\!O}\!_{\infty}(\mathbb{Z}^n_{\operatorname{lex}}),\tau)$ is Hausdorff,
$\{\gamma\}$ is open, and hence isolated.
\end{proof}

Since every \v{C}ech complete space (and hence every locally compact
space) is Baire, Theorem~\ref{theorem-4.1} implies
Corollaries~\ref{corollary-4.2} and \ref{corollary-4.3}.

\begin{corollary}\label{corollary-4.2}
Every Hausdorff \v{C}ech complete (locally compact) topology $\tau$
on $\mathscr{I\!O}\!_{\infty}(\mathbb{Z}^n_{\operatorname{lex}})$ such that
$(\mathscr{I\!O}\!_{\infty}(\mathbb{Z}^n_{\operatorname{lex}}),\tau)$ is a
Hausdorff semitopological semigroup is discrete.
\end{corollary}

\begin{corollary}\label{corollary-4.3}
Every Hausdorff Baire topology (and hence \v{C}ech complete or
locally compact topology) $\tau$ on
$\mathscr{I\!O}\!_{\infty}(\mathbb{Z}^n_{\operatorname{lex}})$ such that
$(\mathscr{I\!O}\!_{\infty}(\mathbb{Z}^n_{\operatorname{lex}}),\tau)$ is a
Hausdorff topological semigroup is discrete.
\end{corollary}

\begin{remark}\label{remark-4.4}
Example~4.4 and Proposition~4.5 from \cite{GutikRepovs2012} show that there exists a non-discrete Tychonoff topology $\tau_W$ on the semigroup
$\mathscr{I}^{\!\nearrow}_{\infty}(\mathbb{Z})$ such that
$(\mathscr{I}^{\!\nearrow}_{\infty}(\mathbb{Z}),\tau_W)$ is a
topological inverse semigroup. Then by Proposition~\ref{proposition-2.3} we get that for every positive integer $n$ there exists a non-discrete Tychonoff topology $\tau_W^n$ on the semigroup
$\mathscr{I\!O}\!_{\infty}(\mathbb{Z}^n_{\operatorname{lex}})$ such that
$(\mathscr{I\!O}\!_{\infty}(\mathbb{Z}^n_{\operatorname{lex}}),\tau_W^n)$ is a
topological inverse semigroup.
\end{remark}

\begin{theorem}\label{theorem-4.5}
Let $n$ be a positive integer and $S$ be a topological semigroup which contains
$\mathscr{I\!O}\!_{\infty}(\mathbb{Z}^n_{\operatorname{lex}})$ as a dense discrete
subsemigroup. If
$I=S\setminus\mathscr{I\!O}\!_{\infty}(\mathbb{Z}^n_{\operatorname{lex}})
\neq\varnothing$ then $I$ is an ideal of $S$.
\end{theorem}

\begin{proof}
Suppose that $I$ is not an ideal of $S$. Then at least one of the
following conditions holds:
\begin{equation*}
    1)~I\cdot\mathscr{I\!O}\!_{\infty}(\mathbb{Z}^n_{\operatorname{lex}})\nsubseteq I,
    \qquad 2)~\mathscr{I\!O}\!_{\infty}(\mathbb{Z}^n_{\operatorname{lex}})\cdot
    I\nsubseteq I,
    \qquad \mbox{or}
    \qquad 3)~I\cdot I\nsubseteq I.
\end{equation*}
Since $\mathscr{I\!O}\!_{\infty}(\mathbb{Z}^n_{\operatorname{lex}})$ is a dense
discrete subspace of $S$, Theorem~3.5.8 from~\cite{Engelking1989}
implies that $\mathscr{I\!O}\!_{\infty}(\mathbb{Z}^n_{\operatorname{lex}})$ is an
open subspace of $S$. Suppose there exist
$\alpha\in\mathscr{I\!O}\!_{\infty}(\mathbb{Z}^n_{\operatorname{lex}})$ and
$\beta\in I$ such that $\beta\cdot\alpha=\gamma\notin I$. Since
$\mathscr{I\!O}\!_{\infty}(\mathbb{Z}^n_{\operatorname{lex}})$ is a dense open
discrete subspace of $S$, the continuity of the semigroup operation
in $S$ implies that there exists an open neighbourhood $U(\beta)$ of
$\beta$ in $S$ such that $U(\beta)\cdot \{\alpha\}=\{\gamma\}$.
Hence we have that
$\big(U(\beta)\cap\mathscr{I\!O}\!_{\infty}(\mathbb{Z}^n_{\operatorname{lex}})\big)
\cdot \{\alpha\}=\{\gamma\}$ and the set
$U(\beta)\cap\mathscr{I\!O}\!_{\infty}(\mathbb{Z}^n_{\operatorname{lex}})$ is
infinite. But by Proposition~\ref{proposition-2.8}, the equation
$\chi\cdot\alpha=\gamma$ has finitely many solutions in
$\mathscr{I\!O}\!_{\infty}(\mathbb{Z}^n_{\operatorname{lex}})$. This contradicts
the assumption that $\beta\in
S\setminus\mathscr{I\!O}\!_{\infty}(\mathbb{Z}^n_{\operatorname{lex}})$. Therefore
$\beta\cdot\alpha=\gamma\in I$ and hence
$I\cdot\mathscr{I\!O}\!_{\infty}(\mathbb{Z}^n_{\operatorname{lex}})\subseteq I$.
The proof of the inclusion
$\mathscr{I\!O}\!_{\infty}(\mathbb{Z}^n_{\operatorname{lex}})\cdot I\subseteq I$ is
similar.

Suppose there exist $\alpha,\beta\in I$ such that $\alpha\cdot
\beta=\gamma\notin I$. Since
$\mathscr{I\!O}\!_{\infty}(\mathbb{Z}^n_{\operatorname{lex}})$ is a dense open
discrete subspace of $S$, the continuity of the semigroup operation
in $S$ implies that there exist open neighbourhoods $U(\alpha)$ and
$U(\beta)$ of $\alpha$ and $\beta$ in $S$, respectively, such that
$U(\alpha)\cdot U(\beta)=\{\gamma\}$. Hence we have that
$\big(U(\beta)\cap\mathscr{I\!O}\!_{\infty}(\mathbb{Z}^n_{\operatorname{lex}})\big)
\cdot
\big(U(\alpha)\cap\mathscr{I\!O}\!_{\infty}(\mathbb{Z}^n_{\operatorname{lex}})\big)
=\{\gamma\}$ and the sets
$U(\beta)\cap\mathscr{I\!O}\!_{\infty}(\mathbb{Z}^n_{\operatorname{lex}})$ and
$U(\alpha)\cap\mathscr{I\!O}\!_{\infty}(\mathbb{Z}^n_{\operatorname{lex}})$ are
infinite. But by Proposition~\ref{proposition-2.8}, the equations
$\chi\cdot\beta=\gamma$ and $\alpha\cdot\kappa=\gamma$ have finitely
many solutions in $\mathscr{I\!O}\!_{\infty}(\mathbb{Z}^n_{\operatorname{lex}})$.
This contradicts the assumption that $\alpha,\beta\in
S\setminus\mathscr{I\!O}\!_{\infty}(\mathbb{Z}^n_{\operatorname{lex}})$. Therefore
$\alpha\cdot\beta=\gamma\in I$ and hence $I\cdot I\subseteq I$.
\end{proof}

\begin{proposition}\label{proposition-4.6}
Let $n$ be a positive integer and $S$ be a Hausdorff topological semigroup which contains
$\mathscr{I\!O}\!_{\infty}(\mathbb{Z}^n_{\operatorname{lex}})$ as a dense discrete
subsemigroup. Then for every
$\gamma\in\mathscr{I\!O}\!_{\infty}(\mathbb{Z}^n_{\operatorname{lex}})$ the set
\begin{equation*}
    D_\gamma=\left\{(\chi,\varsigma)\in
    \mathscr{I\!O}\!_{\infty}(\mathbb{Z}^n_{\operatorname{lex}})
    \times\mathscr{I\!O}\!_{\infty}(\mathbb{Z}^n_{\operatorname{lex}})\mid
    \chi\cdot \varsigma=\gamma\right\}
\end{equation*}
is a closed-and-open subset of $S\times S$.
\end{proposition}

\begin{proof}
Since $\mathscr{I\!O}\!_{\infty}(\mathbb{Z}^n_{\operatorname{lex}})$ is a discrete
subspace of $S$ we have that $D_\gamma$ is an open subset of
$S\times S$.

Suppose that there exists
$\gamma\in\mathscr{I\!O}\!_{\infty}(\mathbb{Z}^n_{\operatorname{lex}})$ such that
$D_\gamma$ is a non-closed subset of $S\times S$. Then there exists
an accumulation point $(\alpha,\beta)\in S\times S$ of the set
$D_\gamma$. The continuity of the semigroup operation in $S$ implies
that $\alpha\cdot\beta=\gamma$. But
$\mathscr{I\!O}\!_{\infty}(\mathbb{Z}^n_{\operatorname{lex}})\times
\mathscr{I\!O}\!_{\infty}(\mathbb{Z}^n_{\operatorname{lex}})$ is a discrete
subspace of $S\times S$ and hence by Theorem~\ref{theorem-4.5}, the
points $\alpha$ and $\beta$ belong to the ideal $I=S\setminus
\mathscr{I\!O}\!_{\infty}(\mathbb{Z}^n_{\operatorname{lex}})$ and hence
$\alpha\cdot \beta\in S\setminus \mathscr{I\!O}\!_{\infty}(\mathbb{Z}^n_{\operatorname{lex}})$ cannot be equal to $\gamma$.
\end{proof}

\begin{theorem}\label{theorem-4.7}
If a Hausdorff topological semigroup $S$ contains
$\mathscr{I\!O}\!_{\infty}(\mathbb{Z}^n_{\operatorname{lex}})$ as a dense discrete
subsemigroup for some positive integer $n$ then the square $S\times S$ cannot be pseudocompact.
\end{theorem}

The proof of Theorem~\ref{theorem-4.7} is similar to that of
Theorem~5.1$(3)$ of \cite{BanakhDimitrovaGutik2010}.

Recall that, a topological semigroup $S$ is called $\Gamma$-compact
if for every $x\in S$ the closure of the set $\{x,x^2,x^3,\ldots\}$
is a compactum in $S$ (see \cite{HildebrantKoch1988}). We recall
that the Stone-\v{C}ech compactification of a Tychonoff space $X$ is
a compact Hausdorff space $\beta X$ containing $X$ as a dense
subspace so that each continuous map $f\colon X\rightarrow Y$ to a
compact Hausdorff space $Y$ extends to a continuous map
$\overline{f}\colon \beta X\rightarrow Y$ \cite{Engelking1989}.

The proof of Corollary~\ref{corollary-4.8} is similar to that of
Corollary~4.9 of \cite{GutikRepovs2012}.

\begin{corollary}\label{corollary-4.8}
Let $n$ be a positive integer. If a topological semigroup $S$ satisfies one of the following conditions: $(i)$~$S$ is compact; $(ii)$~$S$ is $\Gamma$-compact; $(iii)$~the square $S\times S$ is countably compact; $(iv)$~$S$ is a countably compact topological inverse semigroup; or $(v)$~the square $S\times S$ is a Tychonoff pseudocompact space, then $S$ does not contain the semigroup
$\mathscr{I\!O}\!_{\infty}(\mathbb{Z}^n_{\operatorname{lex}})$.
\end{corollary}


\end{document}